\documentclass[a4paper,10pt]{article}

\usepackage[utf8]{inputenc}
\usepackage[english]{babel}

\usepackage[T1]{fontenc}

\usepackage{hyperref}
\usepackage{todonotes}

\title{Modulation equation and SPDEs on unbounded domains}
\author{Luigi Amedeo Bianchi\thanks{Technische Universität Berlin, \emph{Current affiliation:} Università degli Studi di Trento \href{mailto:luigiamedeo.bianchi@unitn.it}{luigiamedeo.bianchi@unitn.it}}  \and
Dirk Bl\"omker\thanks{Universit\"at Augsburg \href{mailto:dirk.bloemker@math.uni-augsburg.de}{dirk.bloemker@math.uni-augsburg.de}} 
\and 
Guido Schneider\thanks{Universit\"at Stuttgart  \href{mailto:guido.schneider@mathematik.uni-stuttgart.de}{guido.schneider@mathematik.uni-stuttgart.de}}
}
\date{\today}

\usepackage{amsthm}
\usepackage{amsmath}
\usepackage{amssymb}
\usepackage{amsfonts}

\usepackage{xfrac}

\allowdisplaybreaks[1]

\newcommand{\formereta}{\varsigma}

\newcommand{\e}{{\mathrm{e}}}
\newcommand{\cL}{{\mathcal{L}}}

 \newcommand{\R}{\mathbb{R}}
 \newcommand{\N}{\mathbb{N}}

\newtheorem{theorem}{Theorem}[section]
\newtheorem{corollary}[theorem]{Corollary}

\newtheorem{definition}[theorem]{Definition}
\newtheorem{lemma}[theorem]{Lemma}
\newtheorem{remark}[theorem]{Remark}

\begin{document}

\maketitle

\begin{abstract}
 We consider the approximation via modulation equations for nonlinear SPDEs on unbounded domains with additive space time white noise.
 Close to a bifurcation an infinite band of eigenvalues changes stability, and we study the impact of small space-time white noise on the dynamics close to this bifurcation. 
 
 As a first example we study the stochastic Swift-Hohenberg equation on the whole real line.
 Here due to the weak regularity of solutions the standard methods for modulation equations fail, and we need to develop new tools to treat the approximation.
 
As an additional result we sketch the proof for local existence and uniqueness of solutions for the stochastic Swift-Hohenberg 
 and the complex Ginzburg Landau equations on the whole real line in weighted spaces that allow for unboundedness at infinity of solutions,
 which is natural for translation invariant noise like space-time white noise.
 We use energy estimates to show that solutions of the Ginzburg-Landau equation are H\"older continuous 
 and have moments in those functions spaces. 
This gives just enough regularity to proceed with the error estimates of the approximation result. 
\end{abstract}

\section{Introduction}
We consider the stochastic Swift-Hohenberg equation on the whole real line.
This is one of the prototypes of pattern forming equations and its first instability is supposed 
to be a toy model for the convective instability in Rayleigh-B\'enard convection.
It is given by 
\begin{equation}
 \label{e:SHintro}
 \partial_t u = -(1+\partial_x^2)^2 u +\nu\varepsilon^2 u -u^3 + \sigma \varepsilon^{\sfrac{3}{2}} \partial_t W
\end{equation}
with space-time white noise $\partial_t W$. 
Here $\nu\in\R$ measures the distance from bifurcation, which scales with $\varepsilon^2$  
and $\sigma\geqslant0$ measures the noise strength that scales with $\varepsilon^{3/2}$, 
for a small $0\leqslant  \varepsilon\ll 1$.
We will see later that the scaling is in such a way that close to the bifurcation both terms have an impact on the dynamics. 

Due to the presence of the noise we run into several problems. 
First, solutions have very poor regularity properties and solutions are at most H\"older continuous. 
Thus we need to consider weaker concepts of solutions like the mild formulation of the equation. 
Moreover, due to translation invariance of the noise solutions are in general immediately unbounded in space,
and we need to work in spaces that do allow for growth of solutions for $|x|\to\infty$.
These weighted spaces are not closed under pointwise multiplication, which is a serious problem in the construction of solutions
due to the cubic nonlinearity.

Our main results show that 
close to the change of stability, i.e.~for small $\varepsilon>0$, solutions of \eqref{e:SHintro} are well approximated 
by a modulated wave\footnote{$c.c.$ denotes the complex conjugate, so that $A+c.c.$ is twice the real part of $A$ }
\begin{equation*}
u(t,x) \approx u_A (t,x) :=  \varepsilon A(\varepsilon^2 t,\varepsilon x) e^{ix} + c.c. ,
\end{equation*}
where the amplitude $A$ solves a so called modulation or amplitude equation, which is in our case a stochastic complex-valued Ginzburg-Landau equation
\[
\partial_T A = 4 \partial^2_X A +\nu A -3A|A|^2 + \partial_T \mathcal{W} 
\]
for some complex-valued space-time white noise $\partial_T \mathcal{W}$ given by the formal derivative of a complex-valued standard cylindrical Wiener process $\mathcal{W}$.
In order to obtain a strong error estimate,
$\mathcal{W}$ is constructed by taking a transformation of the Fourier transform of $W$ around the point $1$ in Fourier space. 
See \cite{BB:15} for the precise details.
\subsection{Modulation equations for deterministic PDEs}
The Ginzburg-Landau equation as an effective amplitude equation for the 
description of  pattern forming systems close to the first instability 
has first been derived in the 1960s by Newell and Whitehead, cf.~\cite{NW69}.
The mathematical justification of this approach beyond pure formal calculations 
has been done by Mielke and Melbourne together with coauthors 
either with the help of a Lyapunov-Schmidt reduction (see \cite{Mi92,Mel98,Mel00}),
or with the construction of special solutions, cf.~\cite{IMD89}. 
Approximation results showing that there are solutions of the pattern-forming system 
which behave as predicted by the Ginzburg-Landau equation has been shown by various authors for instance in 
\cite{CE90,vH91,KSM92,Schn94a,Schn94c,Takac96}. 
Moreover, there are attractivity results by Eckhaus \cite{Eck93} and Schneider \cite{Schn95}, 
showing that every small solution can be described after a certain time by the Ginzburg-Landau equation. 
Various results followed in subsequent years: combining the approximation and attractivity results allows to prove the upper semi-continuity of attractors~\cite{MS:95,Schn99c}, 
shadowing by pseudo-orbits, and global existence results for the pattern-forming systems \cite{Schn94b,Schn99b}.
A number of approximation theorems have been proven in slightly modified situations, such as  the
degenerated case of a vanishing cubic coefficient \cite{BS07}, the Turing-Hopf case description
by mean field coupled Ginzburg-Landau equations  \cite{Schn97}, 
the Hopf bifurcation at the Fourier wave number k = 0 \cite{Schn98}, and 
the time-periodic situation  \cite{SU07}. 
Recently, such results have been established in  case of pattern forming 
systems with conservation laws, too, cf. \cite{HSZ11,SZ13,DSSZ16}.
Let us finally point out that this section is just a brief summary of those of the numerous deterministic results existing in the literature which are most closely related to the one presented here.
\subsection{SPDEs in weighted spaces on unbounded domains}
The theory of higher order parabolic stochastic partial differential equations (SPDEs) 
on unbounded domains with translation invariant additive noise 
like space-time white noise is not that intensively studied in recent years.
while for the wave equation with multiplicative noise there are many recent publications. 
See for example~\cite{Kho14, DaSS09, Dal09}.

First publications for parabolic equations are \cite{Peszat1,Peszat2} but with exponential weights 
and  \cite{Iwata} without using any weights in a distributional space. While \cite{Peszat1} 
can cover the case of real valued Ginzburg-Landau, Swift-Hohenberg is not covered.  

In many publications often only noise with a spatial cut off or a decay condition at infinity is treated, 
as for example by Eckmann and Hairer~\cite{EckHai2001}, where the cutoff is in real and in Fourier space, or by Funaki~\cite{Fun1995}.
Furthermore, Rougemont \cite{Ro02} studied the stochastic Ginzburg-Landau using 
exponentially weighted spaces and relatively simple noise that is white in time, but bounded in space.

In many examples, using trace class noise implies $L^2$-valued Wiener processes 
and thus a decay condition both of solutions and of the noise at infinity. 
This leads to $L^2$-valued solutions, as for example 
by Brzezniak and Li~\cite{BreLi2006} or by Kr\"uger and Stannat~\cite{KruSta2014}, where an integral equation is considered.
 In the next paragraph we will comment on the fact that a decay at infinity rules out the effect we want to study here 
 using modulation equations.

The stochastic Ginzburg-Landau Equation in a weighted $L^2$-space was already studied by Bl\"omker and Han \cite{BH:10}. 
The existence and uniqueness result based on a Galerkin Approximation is briefly sketched there 
and the asymptotic compactness of the stochastic dynamical system is shown.

Recently, several publications treat SPDEs with space-time-white noise in weighted Besov spaces: 
see for example R\"ockner, Zhu, and Zhu
\cite{RoeZhZh:P15} or Mourrat and Weber~\cite{MoWe:P15}. 
They work with the two-dimensional $\Phi^4$-model, which is similar to Ginzburg-Landau and 
where renormalization is needed to give a meaning to the two-dimensional equation with this choice of noise.
In order to construct solutions they consider approximations on the torus and then send the size of the domain to infinity,
which is the method also used in this paper. But the authors work directly in weighted Besov spaces, 
while we show our existence and uniqueness result in spaces with less regularity. 
Moreover the result in Besov spaces relies heavily on properties of the heat-semigroup, 
which do not seem to hold for fourth order operators like the Swift-Hohenberg operator.
For example we will see later, that the operator is not dissipative in weighted $L^p$-spaces, while the Laplacian is.
Thus we cannot derive useful a priori bounds for Swift-Hohenberg in $L^p$-spaces.
This is also the reason that our final approximation result is only valid in a weighted $L^2$-space, 
while the residual is bounded in spaces with H\"older regularity. 

Let us finally remark that spaces without weight like $L^2(\R)$ and the usual Sobolev spaces do not include  constant functions and 
modulated pattern, that do appear close to the bifurcation, and which we want to study here using modulation equations. 
In order to include these special solutions one needs to consider  weighted spaces, see for example~\cite{BLW:13} or~\cite{BBY:16}
for publications treating random attractors.

 \subsection{Modulation equations for SPDEs}

In Bl\"omker, Hairer and Pavliotis \cite{BHP:05}, modulation equations for SPDEs on large domains were treated.
The results are quite similar to the ones presented here, but they hold only on large domains of size proportional to $1/\varepsilon$ for Swift-Hohenberg 
and thus the Ginzburg-Landau equation is posed on a domain of order $1$. 
The main advantage is that one can still work with Fourier series, 
and only finitely many modes change stability at the bifurcation. 
Moreover, solutions of the amplitude equation are not unbounded in space and there is no need to consider weighted spaces.  
The drawback is that various constants depend on the size of the domain and the results do not extend to unbounded domains.

The first results for modulation equations for Swift-Hohenberg on the whole real line were presented by  Klepel, Mohammed, and Bl\"omker 
\cite{MBK:13,KBM:13}. Here the authors used spatially constant noise of a  strength of order $\varepsilon$, which is stronger than the one treated here.
Although the  noise does not appear 
directly in the amplitude equation, due to nonlinear interaction and averaging 
additional deterministic terms appear in the Ginzburg-Landau equation. 
Due to the spatial regularity of the noise, the main advantage  is that one can work in spaces with much more spatial 
regularity than we have to use here.
As a consequence, solutions are still bounded in space and do not grow towards infinity at $|x|\to\infty$.

The key result towards a full result for amplitude equations 
on the whole real line with space-time  white noise 
is by Bianchi and Bl\"omker~\cite{BB:15}. Here the full approximation result  for linear SPDEs, namely the Swift-Hohenberg and Ginzburg-Landau equations without cubic terms,
is established. 
This is very useful in the results presented here, as we use it to approximate 
the stochastic convolutions in the mild formulation.

Let us finally remark that a decay at infinity of the noise and thus of the solution leads to a completely different result.
Under the rescaling in space used to obtain the modulation equation,
we conjecture to finally obtain a point forcing at the origin in the Ginzburg-Landau equation, 
which is an interesting question in itself.

%%%%%%%%%%%%%%%%%%%%%%%%%%%%%%%%%%%%%%%%%%%%%%%%

\subsection{Outline of the paper and main results}

%%%%%%%%%%%%%%%%%%%%%%%%%%%%%%%%%%%%%%%%%%%%%%%
 
In Section \ref{sec:set} we introduce basic notation and especially the weighted spaces we are going to work in.
The weights in these spaces should decay polynomially like 
$|x|^\rho$
for large |$x|$. 

The existence and uniqueness of solutions to the Swift-Hohenberg and the Ginzburg-Landau equations is 
briefly sketched in Section~\ref{sec:exuni} and again at the beginning of Section~\ref{sec:reg}.
The results are straightforward and stated in Theorem~\ref{thm:exSH} and~\ref{thm:exGL},  
for the Swift-Hohenberg and the Ginzburg-Landau equations respectively.

For the proof we use the approximation by finite domains with periodic boundary conditions, and obtain existence in energy type spaces, for example for Swift-Hohenberg in 
\[
L^\infty(0,T,L^2_\varrho) \cap L^2(0,T,H^2_\varrho) \cap L^4(0,T,L^4_\varrho).
\]
We can also obtain weak continuity in time, but we do not need it in our approach.

A key technical point is the regularity result of Corollary \ref{cor:maxregA} in Section \ref{sec:reg}, which roughly states:
\begin{theorem}[Regularity]
For sufficiently smooth initial conditions and all $\eta\in(0,1/2)$
\[ 
  A \in L^\infty(0,T,C^{0,\eta}_\kappa ) \quad \text{for all sufficiently  small }\kappa>0\;.
\]
\end{theorem}
Moreover, we also obtain finite moments of these norms.

The crucial point for the following analysis is that the solution of the amplitude equation 
is H\"older up to exponent almost $1/2$ in space with a polynomial weight $\kappa>0$ that is arbitrarily weak. 
We could also show continuity in time, but we do not need this in the estimates later.

The main idea in the proof is to introduce the standard transformation $B=A-\mathcal{Z}$ 
using the Ornstein-Uhlenbeck (OU) process $\mathcal{Z}$ that solves the stochastic linear Ginzburg-Landau equation and is thus Gaussian. 
The H\"older regularity of $\mathcal{Z}$ is a well known result (see Lemma~\ref{lem:regZ}).
The key idea of the transformation is that $B$ solves a random PDE and we can apply energy type estimates in $L^p$- and $W^{1,p}$-spaces for any $p\geq2$, 
which show that $B$ is more regular than $\mathcal{Z}$ and thus $A$ is as regular as $\mathcal{Z}$.

Let us remark that these $L^p$-estimates are not available for the Swift-Hohenberg equation,  where we only have $L^2$- or $H^1$-estimates.
Thus we do not know how to establish via energy estimates higher regularity for solutions in that case. Moreover bootstrapping in the mild formulation is nontrivial as the cubic is unbounded in the spaces we consider, and the semigroup is not regularising the loss in the weight. As we do not need higher order regularity in our approach, we do not follow this path here.

The main approximation result for the amplitude equation is Theorem \ref{thm:res} in Section \ref{sec:res},  
where we bound the residual defined in \eqref{e:residual} and we have the approximation behaving uniformly for times up to order $\varepsilon^{-2}$ in a weighted $C^0$-space. 

\begin{theorem}[Residual]
Let $A$ be a solution of the amplitude equation \eqref{e:GL} with suitable space-time white noise $\partial_T \mathcal{W}$ and regularity given by the previous theorem.
Then for the approximation $u_A$ defined by the modulated wave  (see \eqref{def:uA}) for all small $\kappa>0$ that
\begin{equation}
\| \text{\rm Res}(u_A)( \varepsilon^{-2}\;\cdot)\|_{C^0_\kappa} =\mathcal{O}(\varepsilon^{\frac32 - 2\kappa }) \;.
\end{equation}
\end{theorem}

In the final Section \ref{sec:app} we carry over the bound for the residual to bounds on the error using standard energy type estimates in a weighted $L^2$-space. 
For details see Theorem \ref{thm:final}, here only roughly stated.

\begin{theorem}[Approximation]
Let $A$ be a solution of the amplitude equation \eqref{e:GL} with suitable space-time white noise $\partial_T \mathcal{W}$ and sufficient regularity given by the theorem above.
For   the modulated wave $u_A$ and any solution  $u$ to the Swift Hohenberg equation~\eqref{e:SH} with initial condition close to $u_A(0)$
we obtain for all $\delta>0$
	\[
		\sup_{[0,T_0\varepsilon^{-2}]}\|u-u_A\|_{L^2_{\varrho,\varepsilon}} =\mathcal{O}(\varepsilon^{1-2\delta})
	\]
\end{theorem}
Here the weight is $\varepsilon$-dependent to reflect the rescaling in space. Note that with this rescaling we have $\|u_A\|_{L^2_{\varrho,\varepsilon}} = \mathcal{O}(\varepsilon^{1/2})$.

\subsection{Outlook}
In the present paper we focus only on the one-dimensional stochastic Swift-Hohenberg equation with cubic nonlinearity forced by space-time white noise. Let us comment here briefly 
on possible extensions of the result.

One key feature of the nonlinearity is the property that is maps the dominant Fourier modes around wave-number $\pm1$ again onto these wave-numbers. 
This yields the replacement of $-u^3$ in Swift-Hohenberg with the corresponding amplitude term $-3A|A|^2$.
Moreover, due to the nonlinear stability of the cubic nonlinearity, we obtained additional regularity.
We believe that other nonlinearities with this property should yield similar results. 

Completely different is the case of {\em{quadratic nonlinearities}}, for instance in an interesting variant of the Swift-Hohenberg equation, but also in convection problems like Rayleigh-Benard.
Although this is solved in the deterministic case, in the stochastic case this is still an open problem. 
For Swift-Hohenberg with quadratic nonlinearity in the approximation
one has to take into account the Fourier modes around wave-number $\pm2$  due to the nonlinear interaction of Fourier-modes. 

Here we focus on  space-time white noise, as 
the estimates for the stochastic convolution term were only done for this noise  in~\cite{BB:15}.
Since we rely on that result for the current one, the approximation proven in this paper is only for space-time white noise.
However, any improvement in the result for the stochastic convolution immediately extends to the result presented in this paper.

The case of {\em{coloured noise}} thus seems to be quite lengthy but straightforward. It restricts to the results in~\cite{BB:15}
and only requires a little attention in the spaces involved. In the limit, due to rescaling of the noise, we expect space-time white noise in the 
amplitude equation at least for suitably smooth and localized spatial correlation function of the noise.
This is going to be dealt with in full detail in an upcoming publication.

Let us finally remark that results beyond one spatial dimension are highly non trivial.
For standard Swift-Hohenberg for example in 2D the whole ring is unstable in Fourier space and any Ginzburg-Landau description is doomed to fail. Even if we add symmetry breaking terms to the equation, such as $\partial_y^2 u$, we would obtain a Ginzburg-Landau approximation, but this would be two dimensional, and we only obtain solutions  
to the stochastic Ginzburg-Landau equation in one spatial dimension. Higher dimension would require renormalization of the equation. 

%%%%%%%%%%%%%%%%%%%%%%%%%%%%%%%%%%%%%%%%%%%%%%%%

 \section{Setting}
 \label{sec:set}
 
%%%%%%%%%%%%%%%%%%%%%%%%%%%%%%%%%%%%%%%%%%%%%%%%

 Consider the stochastic Swift-Hohenberg equation
 \begin{equation}
  \label{e:SH}
  \partial_t u = \cL_\nu u -u^3 +\sigma \varepsilon^{\sfrac{3}{2}} \partial_t W\;, \quad u(0)=u_0 
 \end{equation}

on $\R$ with a standard cylindrical Wiener process $W$ in $L^2(\R)$,
which means that  $\partial_t W$ is space-time white noise, and the operator $\cL_\nu= -(1+\partial_x^2)^2 + \nu\varepsilon^2 $, where $\sigma$ and $\nu$ are constants.
In the following, we will also consider  $\cL_0= -(1+\partial_x^2)^2 $.

Due to lack of regularity \eqref{e:SH} is not defined pointwise, but only in a weak PDE sense.
In order to give a rigorous meaning, we use the standard transformation to a random PDE.
We define the stochastic convolution
\[
Z(t)= \varepsilon^{\sfrac{3}{2}} W_{\cL_\nu}(t) = \varepsilon^{\sfrac{3}{2}}\int_0^t e^{(t-s)\cL_\nu}dW(s)\;.
\]

We will see later that $Z$ is for any $\varrho>0$ in the weighted space $C^0_\varrho$ of continuous functions defined in the next section in  \eqref{e:C0rho}.
Note that $\mathcal{Z}$ is the OU-process corresponding to the Ginzburg-Landau equation.

In order to give a meaning to~\eqref{e:SH} we define $v=u-Z$ and consider weak solutions (in the PDE sense) of  
\begin{equation}
   \label{e:SHt}
     \partial_t v= \cL_\nu v - (v+Z)^3\;, \quad v(0)=u_0\;,
\end{equation}
as in the following definition.

\begin{definition}[Weak solution]
\label{def:weak}
We call an $L^3_{\text{loc}}(\mathbb{R})$-valued  stochastic process $v$ with integrable trajectories a weak solution of  \eqref{e:SHt} 
if for all smooth and compactly supported functions $\varphi$ one has with probability $1$ that for all $t\in[0,T]$
\[
\int_\R v(t)\varphi dx = \int_\R u_0\varphi dx +   \int_0^t\int_\R v(s)  \cL_\nu\varphi dx  -  \int_0^t\int_\R(v(s)+Z(s))^3\varphi dx  \;.  
\] 
\end{definition}

\begin{remark}
A sufficiently regular weak solution of~\eqref{e:SHt} is a mild solution of~\eqref{e:SHt}
given by the integral equation
\[
v(t)=e^{t\cL_\nu}v(0) - \int_0^t e^{(t-s)\cL_\nu}(v(s)+Z(s))^3 ds
\]
which is under the substitution $v=u-Z$ a mild solution of~\eqref{e:SH}. 
This concept is also known as Duhamel's formula or variation of constants.
The equivalence of mild and weak solutions under the assumption of relatively weak regularity can be 
found  for example in~\cite{paz83,lun95}. 
The transfer to our situation is straightforward. 
\end{remark}

\begin{remark}
 The main problem of mild solutions for existence and uniqueness of solutions is the unboundedness of the nonlinear operator, so we cannot use the direct fixed point argument for such solutions.
 For example in weighted spaces  $C^0_\varrho$ of continuous functions with weights decaying to $0$ at infinity,
 which are defined in the next section, the cubic nonlinearity is an unbounded operator on $C^0_\varrho$, as it maps
 $C^0_\varrho$ to~$C^0_{3\varrho}$. We always have to cube the weight, too.
 Thus one can show that the right hand side of the mild formulation can not be a contraction in $C^0_\varrho$.

 Similar problems appear for other weighted spaces, as solutions are allowed to be unbounded in space.
 Thus  later in the paper we use the weak formulation to prove existence and uniqueness, 
 and then the mild formulation to verify error estimates.
\end{remark}

%%%%%%%%%%%%%%%%%%%%%%%%%%%%%%%%%%%%%%%%%%%%%%%%%%%

\subsection{Spaces}

%%%%%%%%%%%%%%%%%%%%%%%%%%%%%%%%%%%%%%%%%%%%%%%%

For $\varrho\in\R$, denote by $C^0_\varrho$ the space of continuous functions $v: \R \to \R$ 
such that the following  norm is finite
\begin{equation}
 \label{e:C0rho}
 \|v\|_{C^0_\varrho} =  \sup_{x\in\mathbb{R}} \Big\{  \frac1{(1+x^2)^{\varrho/2}} |v(x)| \Big\} \;.
\end{equation}
This is a monotone increasing sequence of spaces of continuous functions with growth condition at $\pm\infty$ for $\varrho>0$.
See also Bianchi and Bl\"omker \cite{BB:15}. 

\begin{definition}[Weights]
 \label{def:weight}
 We define for $\varrho>0$ the weight function $w_\varrho(x)=  (1+x^2)^{-\varrho/2}$.
 We also define for $c>0$ the scaled weight function $w_{\varrho,c}(x) = w_\varrho(cx) = (1+c^2x^2)^{-\varrho/2}$.
\end{definition}
We have the following properties 
\begin{equation}\label{eq:propweight}
|w_\varrho'(x)| \leqslant  C w_\varrho(x)\;,
\quad 
|w_\varrho^{(n)}(x)| \leqslant  C_n w_\varrho(x)
\quad\text{and}\quad 
|w_{\varrho,c}^{(n)}(x)| \leqslant  C_n c^n w_{\varrho,c}(x)
\;.
\end{equation}
Moreover, $w_\varrho\in L^1(\R)$ if and only if $\varrho>1$.

Let us remark that we have the equivalence of norms  (see Lemma 2.1 of \cite{BB:15})
\begin{equation}
 \label{e:equivC0}
 c \|v\|_{C^0_\varrho} 
 \leqslant 
 \sup_{L\geqslant 1} \{ L^{-\varrho} \|v\|_{C^0([-L,L])}\}
 \leqslant  C
 \|v\|_{C^0_\varrho},
\end{equation}
with strictly positive constants $c$ and $C$.

We can also define  as in  Bates, Lu, Wang~\cite{BLW:13}
weighted spaces for integrable functions.
\[
L^p_\varrho = \{ u \in L^p_{\text{loc}}(\R) \ :\ uw_\varrho^{1/p} \in L^p(\R)\}
\]
with norm
\[
\|u\|_{L^p_\varrho} =  \Big(\int_\R w_\varrho(x)|u(x)|^p dx\Big)^{1/p} \;.
\]
Moreover, we need weighted Sobolev spaces $W^{k,p}_\varrho$ and $H^k_\varrho:=W^{k,2}_\varrho$ defined by the norm
\[
\|u\|_{W^{k,p}_\varrho} := \Big( \sum_{\ell=0}^k \|\partial_x^\ell u\|^p_{L^p_\varrho} \Big)^{1/p}\;. 
\]
As $1 \leqslant  L^{\varrho} w_\varrho(x)$ on $[-L,L]$, it is easy to check that,
\begin{equation}
 \label{e:Wkpbound}
 \sup_{L\geqslant 1} \{ L^{-\varrho/p} \|v\|_{W^{k,p}([-L,L])}\}
 \leqslant  C \|v\|_{W^{k,p}_\varrho}\;.
\end{equation}
In general this is not an equivalence of norms as the opposite inequality is not true.
Note finally that for $\varrho>1$ we have an integrable weight $w_\varrho\in L^1(\R)$ and thus 
by H\"older inequality for all $k\in\N$, $p\geqslant1$ and $\delta>0$ the embedding
\[
W^{k,p+\delta}_\varrho \subset W^{k,p}_\varrho\;. 
\]
Note that this is false for $\varrho<1$ which includes the case of no weight ($\varrho=0$).

We also 
define weighted H\"older spaces $C^{0,\eta}_\kappa$ of locally H\"older continuous functions such that the following norm is finite:
\begin{equation}
 \label{def:wHoe}
 \|A\|_{ C^{0,\eta}_\kappa} = \sup\{ L^{-\kappa}\|A\|_{C^{0,\eta}[-L,L]} \ : \ L>1\}\;.
\end{equation}
This is the natural space for solutions of the SPDE, as the stochastic convolution $Z$ 
will be in such spaces. See for example Lemma~\ref{lem:regZ} later.

%%%%%%%%%%%%%%%%%%%%%%%%%%%%%%%%%%%%%%%%%%%%%%%%%%%%%%%%%%%%%%

\section{Existence and Uniqueness of solutions}
\label{sec:exuni}

%%%%%%%%%%%%%%%%%%%%%%%%%%%%%%%%%%%%%%%%%%%%%%%%

Here in the presentation we mainly focus  on the Swift-Hohenberg equation and state later 
the analogous result for the Ginzburg-Landau equation without proof.

There are older results on SPDEs in weighted spaces like for example \cite{Peszat1,Peszat2} 
using exponential weights.
While for higher order differential operators both treat only globally Lipschitz nonlinearities, 
\cite{Peszat1} treats only second order differential operators, but can deal with Ginzburg-Landau equation.
But none of the results in the literature seem to fit to the case of Swift-Hohenberg.

Let us also remark that there is already the recent result of Mourrat and Weber \cite{MoWe:P15} for the two-dimensional real-valued Ginzburg-Landau 
(or Allen-Cahn) equation that is  similar from the technical point of view, although it is proven in Besov spaces.

The main result of this section is:

\begin{theorem}
 \label{thm:exSH} 
For all $u_0\in L^2_\varrho$ and $T>0$ with  $\varrho>3$ there is a stochastic process such that $\mathbb{P}$-almost surely 
\[
v\in L^\infty(0,T,L^2_\varrho) \cap L^2(0,T,H^2_\varrho) \cap L^4(0,T,L^4_\varrho) 
\]
and $v$ is a weak solution of \eqref{e:SHt} in the sense of Definition \ref{def:weak}. 
Moreover, for any other such weak solution $\tilde{v}$ we have
\[
\mathbb{P}\Big( \sup_{t\in[0,T]}\|v(t)- \tilde{v}(t)\|_{L^2_\varrho}=0 \Big) =1\;.
\]
\end{theorem}

\begin{remark} \label{rem:rho3}
As we are looking at periodic solutions, the weight $w_\varrho$ with $\varrho>3$ has to decay sufficiently fast, so that it guarantees 
that all boundary terms at $\pm\infty$ arising in integration by parts formula 
in the following proof do all vanish.
\end{remark}

For the relatively straightforward proof of Theorem \ref{thm:exSH} we could follow some ideas of  \cite{BH:10} for the stochastic complex-valued Ginzburg-Landau equation,
where a  Galerkin method based on an orthonormal basis of $L^2_\varrho$ was used.
But here we consider the approximation using finite domains and periodic boundary conditions. 
This is a fairly standard approach also presented in~\cite{MoWe:P15} for the $\Phi^4$-model, which is similar to the Ginzburg-Landau equation.
Nevertheless, the approach of \cite{MoWe:P15} in Besov spaces does not seem to work for the Swift-Hohenberg equation, 
as we are for example not able to establish a priori bounds in Besov spaces.

Consider \eqref{e:SHt} on the domain $[-n,n]$ with  $2n$-periodic boundary conditions and initial condition 
$v^{(n)}(0)=u_0|_{[-n,n]} $ and forcing $Z^{(n)}=Z|_{[-n,n]}$.

By standard  theory of parabolic PDEs (see for example \cite{Robinson:book,Temam:book}), there is for all $n\in\N$ a  $2n$-periodic solution 
\[
v^{(n)}
 \in L^\infty(0,T,L^2([-n,n])) \cap L^2(0,T,H^2([-n,n])) \cap L^4(0,T,L^4([-n,n]))\;,
\]
which extends by periodicity to a solution on the whole real line $\R$ if we extend both $v^{(n)}(0)$ and $Z^{(n)}$ periodically, too.

Using a weight $\varrho>3$, so that all integrals and integrations by parts are well defined,
we obtain:
\begin{align}
\label{e:aprio}
 \frac12\partial_t \|v^{(n)}\|^2_{L^2_\varrho} 
 & = \int_\R w_\varrho v^{(n)} \partial_t v^{(n)} dx \nonumber\\
 &\leqslant   \int_\R w_\varrho v^{(n)} \cL_\nu v^{(n)} dx - \int_\R w_\varrho v^{(n)}(v^{(n)}+Z^{(n)})^3 dx \\
  &\leqslant   \int_\R w_\varrho v^{(n)} \cL_\nu v^{(n)} dx - \frac12 \int_\R w_\varrho |v^{(n)}|^4 dx + C \int_\R w_\varrho |Z^{(n)}|^4 dx \nonumber \;.
\end{align}
Now we first use that $Z$ is uniformly bounded in $C^0_\gamma$ for any small $\gamma>0$. See  Lemma \ref{lem:regZ} below.
Thus the $L^4_\varrho$-norm of $Z$ is finite for $\varrho>1$, and we now want to uniformly bound $\|Z^{(n)}\|^4_{L^4_\varrho}$ in the formula \eqref{e:aprio} above by $\|Z\|^4_{L^4_\varrho}$.
For this first a simple calculation verifies for $k,n\in\mathbb{N}_0$ 
\[
\sup_{|x|\leqslant  n} \frac{w_\varrho(x+2nk)}{w_\varrho(x)} 
= \left[ \sup_{|z|\leqslant 1}\frac{1+z^2n^2}{1+(z+2k)^2n^2}   \right]^{\varrho/2} \leqslant  \left[  \frac{2}{(2|k|-1)^2}  \right]^{\varrho/2} \;.
\]
Thus we obtain by periodicity for $\varrho>1$
\begin{align*}
 \|Z^{(n)}\|^4_{L^4_\varrho}
 = \int_\R w_\varrho |Z^{(n)}|^4 dx 
 & = \sum_{k\in\mathbb{Z}} \int_{(2k-1)n}^{(2k+1)n} w_\varrho(x) |Z(x)|^4 dx \\
  & = \sum_{k\in\mathbb{Z}} \int_{-n}^n w_\varrho(x+2nk) |Z(x)|^4 dx \\
& \leqslant  \sum_{k\in\mathbb{Z}}  2^{\varrho/2}  |2|k|-1|^{-\varrho}   \int_{-n}^n w_\varrho(x) |Z(x)|^4 dx \\
& \leqslant   2^{1+\varrho/2} \sum_{k\in\mathbb{N}_0}   |2k-1|^{-\varrho}  \|Z\|^4_{L^4_\varrho}\;.
\end{align*}

To proceed with \eqref{e:aprio}, we need a bound on the quadratic form of the operator~$\cL_\nu$. For this we use the following Lemma  (compare to Lemma 3.8 of Mielke, Schneider~\cite{MS:95}).

\begin{lemma}
\label{lem:spec}
 For any weight $w_\varrho$ with $\varrho>0$ given by Definition~\ref{def:weight} we have 
 \[
 \int_\R w_\varrho v \cL_0 v\, dx \leqslant   - \frac{C_2}{1+2C_2  } 
 \| v''\|^2_{L^2_\varrho} 
 +  C_2  (3 +\frac52 C_2) \| v\|_{L^2_\varrho}^2 \;.
 \]
\end{lemma}

\begin{remark}
This is not sufficient for the  approximation result later, as we need $C_2 = O(\varepsilon^2)$, which is achieved 
if we consider $w_{\varrho,\varepsilon}$ instead of $w_\varrho$. 
\end{remark}

\begin{proof}
We have to prove the Lemma first for smooth compactly supported or periodic $v$ and then also extend by 
continuity to any $v\in H^2_\varrho$.
Results like this are standard.
In order to not overload the subsequent presentation with indices, we do not recall this approximation step.
The same proof as presented below would hold for the approximation, 
and one just needs to check in the final estimate that we can pass to the limit.  

Integration by parts and H\"older's inequality yield
\begin{multline*}
  \int_\R w_\varrho v \cL_0 v\, dx =\\
  \begin{aligned}
 &= -   \int_\R w_\varrho v^2\, dx - 2 \int_\R w_\varrho v v''\, dx +   \int_\R   w_\varrho' v v''' \, dx +   \int_\R   w_\varrho v' v''' \, dx\nonumber\\
 &=  -   \int_\R w_\varrho v^2\, dx - 2 \int_\R w_\varrho v v''\, dx -   \int_\R   w_\varrho'' v v'' \, dx -2 \int_\R   w_\varrho' v' v'' \, dx - \int_\R   w_\varrho v'' v'' \, dx \nonumber \\
 &=  -   \int_\R w_\varrho v^2\, dx - 2 \int_\R w_\varrho v v''\, dx -   \int_\R   w_\varrho'' v v'' \, dx + \int_\R   w_\varrho'' (v')^2 \, dx - \int_\R   w_\varrho v'' v'' \, dx \nonumber \\
    &\leqslant   -   \| v\|^2_{L^2_\varrho} - \|v''\|^2_{L^2_\varrho}
   -  \int_\R (2 w_\varrho+ w_\varrho'') v v'' \, dx  + C_2\int_\R  w_\varrho |v'|^2 \, dx \nonumber \;.
   \end{aligned}
 \end{multline*}
Now we use the following interpolation inequality
\begin{align*} 
 \int_\R  w_\varrho (v')^2 \, dx 
 & =  -  \int_\R  w_\varrho' vv' \, dx -  \int_\R  w_\varrho vv'' \, dx \\
  & =  \frac12 \int_\R  w_\varrho'' v^2 \, dx -  \int_\R  w_\varrho vv'' \, dx \\
  & \leqslant   \frac{C_2}{2} \| v\|_{L^2_\varrho}^2 + \| v\|_{L^2_\varrho}\| v''\|_{L^2_\varrho} 
\end{align*}
to obtain:
 \begin{multline*} 
 \int_\R w_\varrho v \cL_0 v\, dx\leqslant\\
     \begin{aligned}
     &\leqslant  - (1- \frac{C_2^2}{2})  \| v\|^2_{L^2_\varrho} - \|v''\|^2_{L^2_\varrho}
   -  \int_\R (2 w_\varrho+ w_\varrho''+ C_2 w_\varrho) v v'' \, dx   \nonumber\\
    &\leqslant -  (1- \frac{C_2^2}{2}) \| v\|^2_{L^2_\varrho} - \|v''\|^2_{L^2_\varrho}
   + 2 (1 +  C_2) \| v\|_{L^2_\varrho}\|v''\|_{L^2_\varrho}  \nonumber\\
       &\leqslant -  (1- \frac{C_2^2}{2}-(1+C_2)\delta) \| v\|^2_{L^2_\varrho} -  (1-(1+C_2)\delta^{-1})\|v''\|^2_{L^2_\varrho}  \nonumber\\
 & =   - \frac{C_2}{1+2C_2  } 
 \| v''\|^2_{L^2_\varrho}  
 +  C_2  (3 +\frac52 C_2) \| v\|_{L^2_\varrho}^2  \nonumber \;,
 \end{aligned}
\end{multline*}
where we used Young's inequality with $\delta=1+2C_2$. This finishes the proof of the Lemma.
\end{proof}

Going back to \eqref{e:aprio} and using Lemma \ref{lem:spec} 
we obtain the following result for the $2n$-periodic approximation $v^{(n)}$.

 \begin{lemma}
 \label{lem:apprio}
 Let $u_0$ be in $L^2_\varrho$ for some $\varrho>3$, then 
 there is a small constant $c>0$ and a large constant $C>0$ such that for all $t>0$
 \[
 \partial_t \|v^{(n)}\|^2_{L^2_\varrho}  
 \leqslant  - c \| v^{(n)}\|^2_{H^2_\varrho} - \| v^{(n)}\|_{L^4_\varrho}^4  +  C \| v^{(n)}\|_{L^2_\varrho}^2 + C\| Z\|_{L^4_\varrho}^4 \;.
 \]
 \end{lemma}
 
 As already mentioned this result at least on bounded domains is well known. 
 Usually one would estimate the $v^2-v^4$ by $C-v^2$ in order to obtain bounds on the $L^2$-norm that are uniform in time.
 But as we are only after the existence of solutions in this section, 
 we keep the $L^4$-norm in order to exploit that regularity.

 The following corollary is standard for  a priori estimates as in Lemma \ref{lem:apprio}.
 First  by neglecting 
 the negative terms on the right hand side and by applying Gronwall inequality 
 we obtain an $L^\infty(0,T,L^2_\varrho)$-bound.
 The final two estimates follow by integrating in time the inequality in Lemma \ref{lem:apprio}.

 \begin{corollary}
Under the assumptions of the previous Lemma~\ref{lem:apprio} the sequence $\{v^{(n)}\}_{n\in\N}$ is uniformly bounded 
 in $L^\infty(0,T,L^2_\varrho) \cap L^2(0,T, H^2_\varrho)\cap L^4(0,T,L^4_\varrho)$
for all $T>0$. Moreover,  $\{v^{(n)}\}_{n\in\N}$ is also uniformly bounded 
 in 
 \[L^\infty(0,T,L^2([-L,L])) \cap L^2(0,T, H^2([-L,L]))\cap L^4(0,T,L^4([-L,L]))
 \]
 for all $T>0$ and all $L>0$.
 \end{corollary}

 Now we can finalize the proof of
Theorem \ref{thm:exSH}. 
 By taking consecutively subsequences 
 we obtain a $v \in L^\infty(0,T,L^2_\varrho) \cap L^2(0,T, H^2_\varrho)\cap L^4(0,T,L^4_\varrho)$
 such that 
 \[
v^{(n_k)} \rightharpoonup v
\quad\text{in } L^2(0,T, H^2_\varrho)\cap L^4(0,T,L^4_\varrho)
 \]
 and 
  \[
v^{(n_k)}\stackrel{*}{\rightharpoonup} v
\quad\text{in } L^\infty(0,T,L^2_\varrho)\;.
 \]
Furthermore using a diagonal argument, 
  \[
v^{(n_k)} \rightharpoonup v
\quad\text{in } L^2(0,T, H^2_\text{loc})\cap L^4(0,T,L^4_\text{loc})
 \]
 and 
  \[
v^{(n_k)}\stackrel{*}{\rightharpoonup} v
\quad\text{in } L^\infty(0,T,L^2_\text{loc})\;.
 \]

We now exploit the compactness on bounded intervals (e.g.~via the Aubin-Lions Lemma).
For this we use a uniform bound on $ v^{(n)}$ in $L^2([0,T], H^{2}[-L,L])$ 
together with a bound on the time derivative  $\partial_t v^{(n)}$ in $L^{4/3}([0,T], H^{-2}[-L,L])$, 
which follows from the weak formulation together with the bounds on $v^{(n)}$ already established.   

We then obtain (by taking another subsequence)
\[
v^{(n_k)} \to v \quad\text{in all } L^2([0,T],H^1[-L,L])
\]
and thus 
\[
v^{(n_k)}(t,x) \to v(t,x) \quad\text{for almost all  } (t,x)\;.
\]

Thus, by passing to the limit in the weak formulation for $v^{(n)}$
we obtain for any $\varphi\in C^\infty_c(\R)$ (smooth and compactly supported)
that
\[
\langle v(t) ,\varphi\rangle =  \langle u_0 ,\varphi\rangle 
- \int_0^t \langle (1+\partial_x^2) v(s), (1+\partial_x^2)\varphi \rangle ds 
- \int_0^t \langle ( v(s)+Z(s))^3 , \varphi \rangle ds\;.
\]
This implies  that $v$ is a weak solution of
\[
\partial_t v= \cL_\nu v - (v+Z)^3\;, \quad v(0)=u_0\;.
\]
Furthermore by regularity of $v$ we can take the scalar product with $v$ here. 
This will be used in the proof of uniqueness.

\begin{remark}
One needs to be careful here, as the resulting 
limit (i.e., the solution) is in general not a measurable random variable.
This is well known and  due to the fact that we take subsequences that might depend on the given realization of $Z$, and thus the limit is in general not measurable. 
Of course one could try to construct a measurable selection, but 
uniqueness, which is proved in the next step,  
enforces that the whole sequence $v^{(n)}$ converges to the solution $v$ and thus the limit is anyway a measurable random variable.
\end{remark}

For uniqueness consider two weak solutions $v_1$ and $v_2$ 
in $L^\infty(0,T,L^2_\varrho) \cap L^2(0,T, H^2_\varrho)\cap L^4(0,T,L^4_\varrho)$ with $\varrho>3$.
Define 
\[d=v_1-v_2
\]
which solves 
\[
\partial_t d= \cL_\nu d - (v_1+Z)^3 + (v_2+Z)^3\;, \quad d(0)=0\;.
\]
By the regularity of $d$ we have $\partial_t d \in L^2(0,T,H^{-2}_\varrho)$, thus we can take the $L^2_\varrho$-scalar product with $d$ 
to obtain 
\begin{align*}
 \frac12 \partial_t \|d\|^2_{L^2_\varrho}
 &= \int_\R  w_\varrho d[ \cL_\nu d - (v_1+Z)^3 + (v_2+Z)^3] dx\\
  &= \int_\R  w_\varrho d[ \cL_\nu d - (d+v_2+Z)^3 + (v_2+Z)^3] dx\\
    &= \int_\R  w_\varrho d[ \cL_\nu d - d^3 - 3 d^2(v_2+Z) - 3d(v_2+Z)^2] dx\\
  &\leqslant  \int_\R  w_\varrho d \cL_\nu d dx  -  \int_\R  w_\varrho [d^4 + 3d^2(v_2+Z)^2] -    3 \int_\R  w_\varrho  d^3(v_2+Z) dx\\ 
    &\leqslant   - c \|d\|^2_{H^2_\varrho} + C \|d\|^2_{L^2_\varrho} 
\end{align*}
using that $3d^3b \leqslant  d^4 + \frac94 d^2b^2$ and Lemma \ref{lem:spec}. 
Neglecting now all negative terms and using Gronwall's inequality yields (as $d(0)=0$) that
\[d(t)=0\quad \text{for all }t\geqslant0 \;,
\]
and thus uniqueness of solutions.

%%%%%%%%%%%%%%%%%%%

\section{Additional regularity for the Ginzburg-Landau equation}
\label{sec:reg}

%%%%%%%%%%%%%%%%%%%%%%%%%%%%%%%%%%%%%%%%%%%%%%%%

At the moment we need a very strong weight for the existence and uniqueness of solutions, 
and also related results like the one of \cite{MoWe:P15} always use Besov spaces with integrable weights.
Recall the amplitude equation for the complex-valued amplitude $A$
\begin{equation}
\label{e:GL}
 \partial_T A =4\partial_X^2 A + \nu A -3A|A|^2 + \partial_T \mathcal{W}\;, \qquad A(0)=A_0\;,
\end{equation}
with complex-valued space time white noise $ \partial_T \mathcal{W}$.
Now we use again the standard substitution 
\[
B=A- \mathcal{Z}
\]
with stochastic convolution for $\Delta_\nu = 4\partial_X^2+\nu$ defined by
\begin{equation*}
 \mathcal{Z}(T) =\mathcal{W}_{\Delta_\nu}(T)
 = \int_0^T \e^{(T-s)\Delta_\nu} d \mathcal{W}(s)\;.
\end{equation*}
Now $B$ solves 
\begin{equation}
 \label{e:GLt}
 \partial_T B =4\partial_X^2 B + \nu B -3(B+\mathcal{Z})|B+\mathcal{Z}|^2\;, \qquad B(0)=A_0\;.
\end{equation}
In the regularity results of this section, we will try to weaken the weight as much as possible. 
Moreover, we show spatial H\"older regularity, which is the most we can hope for, 
as we are limited by the regularity of the stochastic convolution~$\mathcal{Z}$. See Lemma \ref{lem:regZ} below.

The key idea is to use energy estimates together with a classical bootstrap argument:
\begin{itemize}
\item Using the $L^2_\varrho$-energy estimate we obtain 
 \[
 A-\mathcal{Z}\in L^\infty(0,T,L^2_\varrho) \cap L^2(0,T,H^1_\varrho) \cap L^4(0,T,L^4_\varrho)
 \]
 in the proof of existence in Theorem \ref{thm:exGL}.
\item Using the $L^p_\varrho$-norm we derive  $A \in L^\infty (0,T,L^q_\varrho)$ in Lemma \ref{lem:5.1}.
\item The $H^1_\varrho$-norm yields $ A-\mathcal{Z} \in L^\infty (0,T,H^1_\varrho) \cap  L^2 (0,T,H^2_\varrho)$ 
in  Lemma \ref{lem:apH1}
\item Sobolev embedding yields H\"older regularity $A\in L^\infty(0,T, C^0_\kappa)$ for arbitrarily small weight  $\kappa>0$. 
See Theorem \ref{thm:apHoeld}.
\item Using the $W^{1,p}_\varrho$-norm we derive $A - \mathcal{Z} \in L^\infty(0,T, W^{1,2p}_\varrho)$
in Lemma \ref{lem:apW1p}.
\item The final result again by Sobolev embedding is $ A \in L^\infty(0,T,C^{0,\eta}_\kappa )$
for all H\"older exponents $\eta\in(0,1/2)$ and for arbitrarily small weight  $\kappa>0$. 
See Corollary \ref{cor:maxregA}.
\end{itemize}

This procedure can only be done for the amplitude equation, but not for the Swift-Hohenberg equation.
For example for the $L^p$-estimate we need that $\langle u^{p-1}, \Delta u \rangle_{L^2_\varrho} \leq c \|u\|_{L^p_\varrho}$, which holds for 
the Laplacian for any $p\geq2$, but
for  the Swift-Hohenberg operator only for $p=2$.

The proof of existence and uniqueness with a strong weight is the same as for the Swift-Hohenberg equation before.
We only need the slightly weaker assumption $\varrho>2$ in this case. 
The precise theorem is:
\begin{theorem}
 \label{thm:exGL} 
For $T>0$ and all $A_0\in L^2_\varrho$ with $\varrho>2$ there is a complex-valued stochastic process $B$ such that $\mathbb{P}$-almost surely 
\[
B\in L^\infty(0,T,L^2_\varrho) \cap L^2(0,T,H^1_\varrho) \cap L^4(0,T,L^4_\varrho) 
\]
and $B$ is a weak solution of \eqref{e:GLt}. 
Moreover, for any other such weak solution $\tilde{B}$ we have
\[
\mathbb{P}( \sup_{t\in[0,T]}\|B(t)- \tilde{B}(t)\|_{L^2_\varrho}=0 ) =1\;.
\]
\end{theorem}
Note that in the previous theorem we only assumed $\varrho>2$ for the weight. 
This is due to the fact that for the Laplacian we need one integration by parts less. See Remark \ref{rem:rho3}.

While for $B$, in the following we can go all the way up to H\"older exponent $1$ and even show $W^{1,p}_\varrho$-regularity.
For the amplitude $A$ we are limited by the following Lemma on the regularity of the stochastic convolution $\mathcal{Z}$.
\begin{lemma}\label{lem:regZ}
For all $\eta<1/2$, $T>0$ and small weight $\kappa>0$ one has $\mathbb{P}$-almost surely  
\[
\mathcal{Z} \in L^\infty(0,T,C^{0,\eta}_\kappa )\;.
\]
Actually, for all $p>0$ there exists a constant $C_p$ such that
\[
	\mathbb{E}[\sup_{[0,T]}\|\mathcal{Z}\|^p_{C^{0,\eta}_\kappa}]\leqslant  C_p.
\]
\end{lemma}

\begin{proof}[Sketch of Proof]
We refrain from giving all the lengthy details of this proof here.
More details on the estimates used can for example be found 
in Lemma~2.4 and Lemma~3.1 in~\cite{BB:15}, where all tools necessary to prove this 
lemma are presented.

The proof for regularity of the stochastic convolution is fairly  standard and based on the proof 
of  the Kolmogorov test for H\"older continuity of stochastic processes.
First note that by H\"older's inequality it is enough to verify the claim for large $p$.
For spatial regularity we consider the embedding of $C^{0,\gamma}([-L,L])$ into 
$W^{\alpha,p}([-L,L])$ for $\gamma<\alpha<1/2$ and $p\to\infty$.  
Then we can use explicit representation of these norms in terms of integrated H\"older quotients.
For the bound in time, we can use the celebrated factorization method of Da Prato, Kwapie\'n and Zabczyk~\cite{DKZ87}. 
\end{proof}

Let us first state a standard energy estimate for the $L^{2p}_\varrho$-norm. 
Here and in all other energy estimates, we need to perform these estimates 
for the approximating sequence from the proof of existence, 
and then pass to the limit later. 

We do not give details of the proof of this lemma and the following ones, as they are very similar to the standard ones for real-valued Ginzburg-Landau (see~\cite{Peszat1}). While this paper was under revision, a new result for the weighted supremum norm with very slowly decaying weights was proven in~\cite{MoiWeb}.
\begin{lemma}\label{lem:5.1}
Let $A$ be such that $B=A-\mathcal{Z}$ is a weak solution of \eqref{e:GLt} given by Theorem \ref{thm:exGL} and fix $T>0$.
If $\varrho>1$ and 
$p$ is such that $A(0)\in L^{2p}_\varrho$, 
then $\mathbb{P}$-almost surely
\[
	A, B\in L^\infty (0,T,L^{2p}_\varrho)\;.
\]
Moreover, for all 
$p\geqslant \sfrac{1}{2}$
there exists a constant 
$C_p$
such that
\[
\mathbb{E}\left[\int_0^T \|A\|^{2p}_{L^{2p}_\varrho}ds\right]\leqslant  C_p, \qquad \mathbb{E}\left[\int_0^T \|B\|^{2p}_{L^{2p}_\varrho}ds\right]\leqslant  C_p.
\]
\end{lemma}

We also need $L^\infty(0,T,H^1_\varrho)$-regularity.
\begin{lemma}
\label{lem:apH1}
Let $A$ be such that $B=A-\mathcal{Z}$ is a weak solution of~\eqref{e:GLt} given by Theorem \ref{thm:exGL} and fix $T>0$.
If $\varrho>2$ and $A(0)\in H^1_\varrho \cap L^6_\varrho$,
then we have $\mathbb{P}$-almost surely
\[
B= A-\mathcal{Z} \in L^\infty (0,T,H^1_\varrho) \cap  L^2 (0,T,H^2_\varrho)\;.
\]
Moreover, there exists a constant $C$ such that
\[
\mathbb{E}[\int_0^T \|B\|^2_{H^2_\varrho}ds]\leqslant  C, \qquad \mathbb{E}[\sup_{[0,T]}\|B\|^2_{H^1_\varrho}]\leqslant  C\;.
\]
\end{lemma}

Now we are aiming at space-time regularity  for $B$ and thus $A$ in $L^\infty(0,T,C^0_\kappa)$.
This is achieved by a pointwise interpolation of $C^0([-L,L])$ between $L^p([-L,L])$ and $H^1([-L,L])$ 
for each fixed $t\in[0,T_0]$.
But we need to be careful in the arguments as the interpolation constants  depend on the spatial domain size $L$. 
We will prove:
\begin{theorem}
\label{thm:apHoeld}
Let $A$ be such that $B=A-\mathcal{Z}$ is a weak solution of \eqref{e:GLt} given  for $T>0$ by Theorem \ref{thm:exGL}.
If $\varrho>2$, $A(0)\in H^1_\varrho$ and
$A(0)\in L^p_\varrho$ for all large~$p$, 
then   $\mathbb{P}$-almost surely
\[ A, B \in L^\infty(0,T, C^0_\kappa)\quad \text{for all small }\kappa>0\;.
\]
Moreover, for all $p>0$ there exists a constant $C_p$ such that
\[
\mathbb{E}[\sup_{[0,T]}\|A\|^p_{C^{0}_\kappa}]\leqslant  C_p\;,
\qquad \mathbb{E}[\sup_{[0,T]}\|B\|^p_{C^{0}_\kappa}]\leqslant  C_p\;.
\]
\end{theorem}

\begin{proof}
For a bounded interval $I=[-1,1]$ we use for $1/2 > \alpha >1/p$ first the Sobolev embedding of  $W^{\alpha,p}(I)$ into $C^0(I)$,
then interpolation between  $W^{\alpha,p}$ spaces and finally 
the Sobolev embedding of  $H^1(I)$ into $W^{1/2 ,p}(I)$ for all $p\in(1,\infty)$ 
to obtain 
\[
\|A\|_{C^0(I)} \leqslant  C \|A\|_{W^{\alpha,p}(I)} \leqslant  C  \|A\|_{L^p(I)}^{1-2\alpha}\|A\|_{W^{1/2 ,p}(I)}^{2\alpha}
\leqslant  C  \|A\|_{L^p(I)}^{1-2\alpha}\|A\|_{H^1(I)}^{2\alpha}.
\]
Now we use the scaling for $L\geqslant 1$ in order to derive the precise scaling in the domain size $L$
of the constant in the Sobolev embedding.
\[
	\begin{split}
		 \|B(L \cdot )\|_{L^p(I)} &=   \Big(\int_I |B(L x )|^p\; dx\Big)^{1/p} =  \Big(\frac1L \int_{-L}^L |B(x )|^p\; dx\Big)^{1/p} \\
		 & =  L^{-1/p} \|B \|_{L^p([-L,L])} \;.
	\end{split}
 \]
Thus we obtain 
 \begin{align*}
   \|B(L \cdot )\|_{H^1(I)}
   & = 
 L  \|B'( L \cdot )\|_{L^2(I)} +  \|B( L \cdot )\|_{L^2(I)}\\  
 &=L^{1/2}  \|B'\|_{L^2([-L,L])} + L^{-1/2} \|B\|_{L^2(I)} \leqslant   L^{1/2}  \|B\|_{H^1([-L,L])}\;.
 \end{align*}
 Moreover,  we obtain using \eqref{e:Wkpbound}
 \begin{equation*}
\begin{split}
  \|B\|_{C^0([-L,L])} & =    \|B(L\cdot)\|_{C^0([-1,1])} \\
  &\leqslant   C  \|B(L\cdot)\|_{L^p([-1,1])}^{1-2\alpha}\|B(L\cdot)\|_{H^1([-1,1])}^{2\alpha} \\
    &\leqslant   C L^{-\frac1p(1-2\alpha)}  \|B\|_{L^p([-L,L])}^{1-2\alpha}  L^{\frac12(2\alpha)}  \|B\|_{H^1([-L,L])}^{2\alpha} \\
   &\leqslant   C L^{-\frac1p(1-2\alpha)}   L^{\alpha} L^{\varrho(1-2\alpha)/p}  \|B\|_{L^p_\varrho}^{1-2\alpha}  L^{2\varrho\alpha/2}  \|B\|_{H^1_\varrho}^{2\alpha} \\
  &\leqslant   C L^{\frac1p(1-2\alpha)(\varrho-1)} L^{\alpha(1+\varrho)}   \|B\|_{L^p_\varrho}^{1-2\alpha}   \|B\|_{H^1_\varrho}^{2\alpha}  \;.
\end{split}
 \end{equation*}
Now we can first choose $\alpha >0$ small and  then $p>1/\alpha$ sufficiently large,  so that for any given $\kappa>0$ we have a $C>0$
so that 
\begin{equation*} 
\|B\|_{C^0([-L,L])} 
\leqslant 
 C L^\kappa  \|B\|_{L^p_\varrho}^{1-2\alpha}   \|B\|_{H^1_\varrho}^{2\alpha} \;.
\end{equation*}
Thus the claim for $B$ follows from the equivalent definition of the $C^0_\kappa$-norm (see~\eqref{e:equivC0}).
For $A=B-\mathcal{Z}$ we just use the fact that the stochastic convolution $\mathcal{Z}$ is more regular.

We can get the bounds for all sufficiently large moments from 
\[\|B\|_{C^0_\kappa} 
\leqslant  C  \|B\|_{L^p_\varrho}^{1-2\alpha}   \|B\|_{H^1_\varrho}^{2\alpha}
\]
by carefully choosing $\alpha>0$ sufficiently small and taking the supremum in $[0,T]$,
as we have any moments in $L^p_\varrho$ and second moments in $H^1_\varrho$.
\end{proof}
For H\"older continuity of $B$ and thus $A$, we need to proceed with the bootstrap argument
and show $W^{1,2p}_\varrho$-regularity for $B$ first. 
We omit the proof in this case, too. It is based on standard energy estimates, but now for $\|B'\|^{2p}_{L^{2p}}$.
\begin{lemma}
\label{lem:apW1p}
Let $A$ be such that $B=A-\mathcal{Z}$ is a weak solution of \eqref{e:GLt} given for $T>0$ by Theorem \ref{thm:exGL}.
If $\varrho>2$, $p>2$,  $A(0)\in W^{1,2p}_\varrho$ and
$A(0)\in L^{6p}_\varrho$,
then  
\[ 
B= A - \mathcal{Z} \in L^\infty(0,T, W^{1,2p}_\varrho)\;.
\]
Moreover for all $p>2$ there exists a constant $C_p$ such that
\[
\mathbb{E}[\sup_{[0,T]}\|B\|^{2p}_{W^{1,2p}_\varrho}]\leqslant  C_p.
\]
\end{lemma}

Now we turn to the regularity in weighted H\"older spaces defined in 
\eqref{def:wHoe}.

\begin{theorem}\label{thm:whreg}
Let $B=A-\mathcal{Z}$ with $A$ a weak solution of \eqref{e:GLt} given for $T>0$ by Theorem \ref{thm:exGL}. 
If for some $\varrho>2$ and all sufficiently large $p\geqslant 2$ we have  $A(0)\in W^{1,p}_\varrho $   
then  for all $\eta\in(0,1)$
\[ 
  B \in L^\infty(0,T,C^{0,\eta}_\kappa ) \quad \text{for all sufficiently  small }\kappa>0\;.
\]
Moreover, for all $p>0$ there exists a constant $C_p$ such that
\[
\mathbb{E}[\sup_{[0,T]}\|B\|^{p}_{C^{0,\eta}_\kappa}]\leqslant  C_p.
\]
\end{theorem}
Using Lemma~\ref{lem:regZ} for the regularity of $\mathcal{Z}$ we obtain:
\begin{corollary}
\label{cor:maxregA}
 Under the assumptions of the previous theorem for all $\eta\in(0,1/2)$
\[ 
  A \in L^\infty(0,T,C^{0,\eta}_\kappa ) \quad \text{for all sufficiently  small }\kappa>0\;.
\]
Moreover, for all $p>0$ there exists a constant $C_p$ such that
\[
\mathbb{E}[\sup_{[0,T]}\|A\|^{p}_{C^{0,\eta}_\kappa}]\leqslant  C_p.
\]
\end{corollary}

\begin{proof}(Proof of Theorem~\ref{thm:whreg})
We proceed
by using the Sobolev embedding of $W^{\alpha,p}([-L,L])$ into $C^{0,\alpha-1/p}([-L,L])$ for $p+1>\alpha p>1$ and then an interpolation inequality.
As before, we need to take care of the scaling of the constants with respect to $L$. Recall that $I=[-1,1]$. First,
\begin{equation*}
\|B\|_{C^{0,\alpha-1/p}(I)} 
\leqslant  C\|B\|_{W^{\alpha,p}(I)}
\leqslant  C\|B\|^{1-\alpha}_{L^p(I)}\|B\|^{\alpha}_{W^{1,p}(I)}\;.
\end{equation*}
Rescaling for $L\geqslant1$ yields
\begin{equation*}
\begin{split}
\|B\|_{C^{0,\eta}[-L,L]}&=\sup_{\xi,\zeta\in [-L,L]}\frac{|B(\xi)-B(\zeta)|}{|\xi-\zeta|^\eta}+\|B\|_{C^0[-L,L]}\\
&=\sup_{x,y\in I} \frac{|B(xL)-B(yL)|}{|x-y|^\eta L^\eta}+\|B(\cdot L)\|_{C^0(I)}\\
 &\leqslant  L^{-\eta}\|B(\cdot L)\|_{C^{0,\eta}(I)} + \|B(\cdot L)\|_{C^0(I)} \;.
\end{split}
\end{equation*}
Now we take $\formereta=\alpha-1/p$ (recall $\alpha p>1$) and interpolate:
\begin{equation*}
\|B\|_{C^{0,\formereta}[-L,L]} \leqslant  L^{-\formereta} \|B(\cdot L)\|^\alpha_{W^{1,p}(I)}\cdot\|B(\cdot L)\|^{1-\alpha}_{L^{p}(I)}  + \|B(\cdot L)\|_{C^0(I)} \;.
\end{equation*}
Now we rescale back all the norms to the original length scale. For the first one we obtain
\begin{equation*}
\begin{split}
 \|B(\cdot L)\|^p_{W^{1,p}(I)}&= L^p\int_I |B'(Lx)|^p dx + \int_I|B(Lx)|^p dx\\
 &= L^{p-1}\int_{-L}^{L}|B'|^p dz + \frac{1}{L} \int_{-L}^{L}|B|^p dx
 \leqslant  L^{p-1}\|B\|^p_{W^{1,p}[-L,L]}
\end{split}
\end{equation*}
and thus
\[
  \|B(\cdot L)\|^\alpha_{W^{1,p}(I)}\leqslant  L^{\frac{p-1}{p}\alpha}\|B\|^\alpha_{W^{1,p}[-L,L]} \;.
\]
For the second norm in the interpolated part we have, by a substitution,
\[
\|B(\cdot L)\|^p_{L^p(I)}\leqslant  L^{-1}\|B\|^p_{L^p[-L,L]}
\quad \text{and}\quad
  \|B(\cdot L)\|^{1-\alpha}_{L^{p}(I)} \leqslant  L^{-\frac{1-\alpha}{p}}\|B\|^{1-\alpha}_{L^{p}[-L,L]} \;.
\]
We can now put these estimates together and derive
\begin{equation*}
  \begin{split}
	\| B \|_{C^{0, \formereta} [- L, L]} 
	& \leqslant  L^{- \formereta}  \| B (\cdot L)\|^{\alpha}_{W^{1, p}(I)}  \| B (\cdot L) \|^{1 - \alpha}_{L^p(I)}+ \|B(\cdot L)\|_{C^0[-1,1]}\\
	& \leqslant  L^{- \formereta} L^{\frac{p - 1}{p} \alpha} \| B\|^{\alpha}_{W^{1, p} [- L, L]} \cdot L^{- \frac{1 - \alpha}{p}} \| B \|^{1- \alpha}_{L^p [- L, L]} + \|B\|_{C^0[-L,L]}\\
	& =  \| B \|^{\alpha}_{W^{1, p} [- L, L]} \cdot \| B \|^{1 - \alpha}_{L^p[- L, L]} + \|B\|_{C^0[-L,L]}
	\end{split}
\end{equation*}
with $\formereta = \alpha -\frac{1}{p}$, 
where $\formereta \in (0, 1)$, $\alpha \in (\formereta, 1)$ and $p > 1/\alpha$ sufficiently large. It is somewhat remarkable here that the constant above is $1$ and thus independent of $L$.

Now by \eqref{e:Wkpbound} we can change to the weighted space to
obtain
\begin{equation*}
\begin{split}
	\| B \|_{C^{0, \formereta} [- L, L]} 
	& \leqslant  \| B \|^{\alpha}_{W^{1, p} [-L, L]}  \| B \|^{1 - \alpha}_{L^p [- L, L]} + \|B\|_{C^0[-L,L]}\\
	& \leqslant   \| B \|^{\alpha}_{W^{1, p}_{\varrho}}  \| B \|^{1 -\alpha}_{L^p_{\varrho}} L^{\varrho / p} + \|B\|_{C^0[-L,L]} \;.
\end{split}
\end{equation*}
Now  for $\formereta \in (0, 1)$, $\formereta =
\alpha - \frac{1}{p}$, $\alpha \in (\formereta, 1)$ and $p > 1/\alpha$,
using the definition of the weighted H\"older norms from \eqref{def:wHoe} 
and the equivalent representation of the $C^0_\kappa$-norm (see \eqref{e:equivC0})  we  derive
\begin{equation*}
\begin{split}
\|B\|_{C^{0,\formereta}_\kappa} &= \sup_{L\geqslant1}\{L^{-\kappa}\|B\|_{C^{0,\formereta}[-L,L]}\}\\
&\leqslant   \sup_{L\geqslant1}\{L^{\varrho/p-\kappa}\}\| B
\|^{\alpha}_{W^{1, p}_{\varrho}}  \| B \|^{1 - \alpha}_{L^p_{\varrho}} + \|B\|_{C^0_\kappa},
\end{split}
\end{equation*}
and as soon as we choose $p$ large enough that $\varrho/p-\kappa\leqslant  0$ we have finished the proof.
Once again the bounds on all the moments follow easily, as we have all moments of the terms on the right hand side.
\end{proof}

%%%%%%%%%%%%%%%%%%%%%%%%%%%%%%%%%%%%%%%%%%%%%%%%%%%%%%%%%%%%%%%5

\section{Residual}
\label{sec:res}

%%%%%%%%%%%%%%%%%%%%%%%%%%%%%%%%%%%%%%%%%%%%%%%%%%%%%%%%%%%%%%%%
Define the approximation\footnote{recall that $c.c.$ denotes the complex conjugate}
\begin{equation}
\label{def:uA}
u_A(t,x) = \varepsilon A(\varepsilon^2 t, \varepsilon x)\e^{ix} + c.c. \;,
\end{equation}
 where $A$ is both a weak and a mild solution of the amplitude equation given by 
\begin{equation}
 \label{e:GLmild}
 A(T) =  \e^{(4\partial_X^2+\nu) T } A(0) - 3\int_0^T  \e^{(4\partial_X^2+\nu)(T-s) } A|A|^2(s) ds + \mathcal{Z}(T) \;.
\end{equation}
Define the residual
\begin{equation}
 \label{e:residual}
 \text{Res}(\varphi)(t) = \varphi(t) - \e^{t\cL_\nu}\varphi(0)  + \int_0^t \e^{(t-s)\cL_\nu} \varphi(s)^3 ds + \varepsilon^{3/2} W_{\cL_\nu} (t) \;,
\end{equation}
which measures how close $u_A$ is to being a mild solution of Swift-Hohenberg.
In this section we bound $\text{Res}(u_A)$. 
This is a key result to prove the error estimate later. 
First we plug in the definition of $u_A$ to obtain
\begin{align*}
\text{Res}(u_A)(t,\cdot) 
 = {}& \varepsilon \Big[A( \varepsilon^2 t, \varepsilon \cdot)\e_1 - \e^{t\cL_\nu}[A(0,\varepsilon \cdot)\e_1] 
\\&
+  3\varepsilon^2 \int_0^t \e^{(t-s)\cL_\nu} A|A|^2(\varepsilon^2 s ,\varepsilon \cdot) \e_1 ds \Big]\\
& + \varepsilon^3 \int_0^t \e^{(t-s)\cL_\nu}A^3(\varepsilon^2 s , \varepsilon \cdot )\e_3  ds
+ c.c. - \varepsilon^{3/2} W_{\cL_\nu} (t,\cdot)\;,
\end{align*}
where we used the notation 
\[
\e_n(x)=\e^{inx}\;.
\]

Rescaling to the slow timescale, we find
\begin{equation}
\label{e:Res1}
\begin{split}
   \text{Res}(u_A)(T\varepsilon^{-2}, \cdot) 
   = {}& \varepsilon A( T,\varepsilon \cdot )\e_1 - \varepsilon \e^{T\varepsilon^{-2}\cL_\nu}[A(0,\varepsilon \cdot)\e_1]  
   \\ &
  +  \varepsilon \int_0^{T} \e^{(T-s)\varepsilon^{-2}\cL_\nu} \left[A^3(s, \varepsilon \cdot)\e_3 + 3A|A|^2(s, \varepsilon \cdot )  \e_1\right] ds 
    \\ &
    + c.c. - \varepsilon^{3/2} W_{\cL_\nu} (T\varepsilon^{-2},\cdot)\;.
    \end{split}
\end{equation}
Now we need to transform this to obtain the mild formulation \eqref{e:GLmild}.
This will remove all the terms of order $\mathcal{O}(\varepsilon)$.

%%%%%%%%%%%%%%%%%%%%%%%%%%%%%%%%%%%

\subsection{Stochastic Convolution}

%%%%%%%%%%%%%%%%%%%%%%%%%%%%%%%%%%%%%%%%%%%%%%%%

The stochastic convolution in \eqref{e:Res1} can be replaced by (see \cite{BB:15})
\[ 
\varepsilon^{1/2} W_{\cL_\nu} (t,x)  \approx  \mathcal{Z}(\varepsilon^2 t,\varepsilon x)e_1 +c.c.  \quad\text{uniformly on $[0,T\varepsilon^{-2}]$ in } C^0_\gamma  \;.
\]
Here we need that $\mathcal{W}$ is defined as a suitable rescaling of the Fourier-transform of $W$ around $1$ in Fourier space. 
For details we refer to \cite{BB:15}. Here we always assume that $\mathcal{W}$ is a complex-valued standard cylindrical Wiener process such 
that the approximation result for the stochastic convolution holds.
The precise error estimate from \cite{BB:15} is 
\begin{theorem}
For all $T>0$, for all $\kappa >0$, for all $p>1$ and all sufficiently small  $\gamma>0$ 
there is a constant $C>0$ such that
\begin{equation*}
\mathbb{P}\Big(\sup_{[0,T\varepsilon^{-2}]}
\|\varepsilon^{ 1 / 2} W^{(\varepsilon)}_{\mathcal{L}_{\varepsilon}} (t, x) - [ 
    \mathcal{Z} (t \varepsilon^2, x \varepsilon) e_1 +
   \textrm{c.c}.] \|_{C^0_{\gamma}}\geqslant\varepsilon^{1-\kappa}\Big)\leqslant  C \varepsilon^p
\end{equation*}
for all $\varepsilon\in (0,\varepsilon_0)$.
\end{theorem}
We use the following shorthand notation in order to reformulate this lemma.
\begin{definition} 
	\label{def:0}
	We say that a real valued stochastic process $X$ is 
 $\mathcal{O}(f_\varepsilon)$ with high probability, if there is a constant $C>0$ such that for all $p>1$ there is a constant $C_p>0$ such that for $\varepsilon\in(0,1)$
\[ \mathbb{P}\Big(\sup_{[0,T]} |X|  \geqslant C f_\varepsilon\Big) \leqslant  C_p\varepsilon^p\;.
\]
\end{definition}
\begin{lemma} 
\label{lem:stoch}
We can write
	\[
		\varepsilon^{3/2} W_{\cL_\nu} (t,\cdot) = \varepsilon \mathcal{Z}(\varepsilon^2 t, \varepsilon\cdot)\e_1 +c.c. + E_s(t)
	\]
	where $\|E_s(\varepsilon^{-2}\cdot)\|_{C^0_{\gamma}}=\mathcal{O}(\varepsilon^{1-\kappa})$  for any $\kappa>0$ in the sense of the previous definition.
\end{lemma}
%

%%%%%%%%%%%%%%%%%%%%%%%%%%%%%%%%%%%%%%%%

\subsection{Exchange Lemmas and estimates for the residuals}

%%%%%%%%%%%%%%%%%%%%%%%%%%%%%%%%%%%%%%%%%%%%%%%%

Let us now come back to \eqref{e:Res1}. In the following we present lemmas to exchange the Swift-Hohenberg 
semigroup generated by $\cL_\nu=\cL_0 + \varepsilon^2 \nu$ with 
the Ginzburg-Landau semigroup generated by $\Delta_\nu=4\partial_x^2 + \nu$.
The first one is for the initial condition, which is the most difficult one, as we cannot allow for a pole in time.
The second one is for the term in \eqref{e:Res1} that contains $A|A|^2$, while the third one shows that the term in  \eqref{e:Res1} with $A^3$ is negligible.
After applying all the exchange Lemmas to  \eqref{e:Res1} we will see that in \eqref{e:ResExch} below
all of the remaining terms of order $\mathcal{O}(\varepsilon)$ 
will cancel due to the mild formulation of the Amplitude Equation in \eqref{e:GLmild}.

We will state all lemmas here and first prove the bound on the residual, before verifying the Exchange Lemmas.

\begin{lemma}[Exchange Lemma - IC]
\label{lem:lemma1} Let $A_0\in C_\kappa^{0,\alpha}$ for some $\alpha\in(\frac12,1)$ and sufficiently small $ \kappa>0$.  
	Then the  following holds 
	\[
		e^{t\cL_\nu}[A_0(\varepsilon \cdot)\e_1]=(e^{\Delta_\nu T}A_0)(\varepsilon \cdot)\e_1 + E_1(A_0)
	\]
	where $T=\varepsilon^2 t$, and the error $E_1$ is bounded uniformly in time for all small $\kappa>0$ by
	\[
		\|E_1\|_{C^0_\kappa} \leqslant  C \varepsilon^{\alpha-\frac12} \|A_0\|_{ C_\kappa^{0,\alpha}}\;.
	\]
\end{lemma}

\begin{remark}
Here we allow some dependence on higher norms of the initial conditions, i.e.~we assume more regularity for the initial conditions in order to avoid the pole in time that appears 
in the exchange lemma below.
\end{remark}
\begin{remark}
For the solution $A$ of the amplitude equation we showed in Section~\ref{sec:reg} that it splits into a more regular part $B$ and the Gaussian $\mathcal{Z}$.
The process $B$ has $W^{1,p}_\varrho$-regularity as assumed for the initial condition $A(0)=A_0$ in the previous exchange Lemma.
The term $\mathcal{Z}$ is less regular, but thanks to the fact that it is Gaussian, we can still prove the exchange Lemma for initial conditions $A(0)$ which split into the more regular and the Gaussian part  (see the application of Lemma 3.5 in the proof of Theorem 4.2 of \cite{BB:15}.

Thus for our result we could take initial conditions that are as regular, as  the solution of the amplitude equation, but
in order to simplify the statement of the result we refer from adding the less regular Gaussian part here.
\end{remark}
The following lemma is applied in \eqref{e:Res1} to the term in the residual associated to the nonlinearity $|A|^2A$,
in order to exchange the semigroups there.

\begin{lemma}[Exchange Lemma I]
\label{lem:lemma2}
	For any  function $D \in C^{0,\alpha}_\kappa$ with small $\kappa>0$ and $\alpha\in(0,1)$, we have for all $t \in[0,T_0\varepsilon^{-2}]$
	\[
		e^{t\cL_\nu}[D(\varepsilon \cdot)\e_1] = (e^{\Delta_\nu T}D)(\varepsilon \cdot)\e_1 + E_2(T,D)
	\]
	with $T=\varepsilon^2t$ where the error term $E_2$ is bounded by
	\[
	\|E_2(T,D) \|_{C^{0}_\kappa} \leqslant  C \left[(\varepsilon^{2\gamma-1}+\varepsilon^{\alpha-\kappa} )T^{-1/2}+\varepsilon T^{-3/4}\right]  \|D\|_{C^{0,\alpha}_\kappa}\;.
	\]
\end{lemma}
The third lemma is needed in \eqref{e:Res1} for the term in the residual associated to $A^3$, which should be small.
\begin{lemma}[Exchange Lemma II]
\label{lem:lemma3}
	For any  function $D \in C^{0,\alpha}_\kappa$ for all $\kappa>0$ and $\alpha\in(0,1)$, we have for all $t \in[0,T_0\varepsilon^{-2}]$
	\[
		e^{t\cL_\nu}[D(\varepsilon \cdot)\e_3] = E_3(T,D) 
	\]	
	where for any $\gamma\in[1/8,1/2)$ 
	the error term $E_3$  on the slow timescale  $T=\varepsilon^2t$
	is bounded by 
	\[\|E_3(T, D)\|_{C^{0}_\kappa} \leqslant  C ( \varepsilon^{2\gamma-1/2}+ \varepsilon^{\alpha-\kappa})T^{-1/2}
	\|D\|_{C^{0,\alpha}_\kappa}\;.
	\]
\end{lemma}

Now we apply all Exchange Lemmas \ref{lem:lemma1}, \ref{lem:lemma2}, and \ref{lem:lemma3} 
together with the result for the stochastic convolution from Lemma \ref{lem:stoch} 
to the definition of $\text{Res}(u_A)$ from~\eqref{e:Res1} to obtain
\begin{equation}
\label{e:ResExch}
\begin{split}
	\text{Res}(u_A)(t)  ={}& \varepsilon \Big[A(\varepsilon \cdot, \varepsilon^2 t ) - \e^{\Delta_\nu T}[A(\varepsilon \cdot, 0)]
	\\ &
	+  3\varepsilon^2 \int_0^t \e^{(T-\varepsilon^2s)\Delta_\nu} A|A|^2(\varepsilon^2 s,\varepsilon \cdot ) ds - \mathcal{W}_{\Delta_\nu}(\varepsilon^2 t, \varepsilon\cdot)\Big]\e_1 + c.c.\\
	& + \varepsilon [ E_1(A_0)+ E_s]  
	+ \varepsilon \int_0^T  \left[E_2\left(T-S, A|A|^2\right)+ E_3\left(T-S, A^3\right)\right] \; dS.
\end{split}
\end{equation}

With this representation we are done. By substituting $S=s\varepsilon^2$ in the integral, 
we obtain that the whole bracket $[\cdots]\e_1$ is the mild formulation of the Ginzburg-Landau equation (see~\eqref{e:GLmild}) and thus cancels.
Using the bounds on the error terms and the regularity of $A$, we obtain the main result.

Note that all poles from the error terms are integrable and that we choose $\alpha, \gamma<1/2$, arbitrarily close to $1/2$.

\begin{theorem}[Residual]
\label{thm:res}
Let $A$ be a solution of the amplitude equation \eqref{e:GL} and assume that there is a  $\varrho>2$ such that for all $p>1$ one has $A(0)\in W^{1,p}_{\varrho}$.
Then for the approximation $u_A$ defined in \eqref{def:uA} and the residual defined in \eqref{e:residual}  we have
for all small $\kappa>0$ that
\begin{equation}
\label{e:thmmain}
\| \text{\rm Res}(u_A)( \varepsilon^{-2}\;\cdot)\|_{C^0_\kappa} =\mathcal{O}(\varepsilon^{\frac32 - 2\kappa }) \;.
\end{equation}
\end{theorem}

\begin{remark}
	Note that under the assumptions of the previous theorem 
	by the regularity results in Section \ref{sec:reg} we have for all small $\kappa>0$, $\gamma\in(0,1)$ and $\alpha\in(0,\frac12)$ that $A(0)\in C^{0,\gamma}_\kappa$  and $A\in L^\infty([0,T], C^{0,\alpha}_\kappa)$
\end{remark}

 We remark without proof that one could replace   $- 2\kappa $ on the right hand side of 
 \eqref{e:thmmain} by an arbitrarily small $\delta>0$. But as $\kappa$ is small also,
 we state this simpler but weaker statement.

%%%%%%%%%%%%%%%%%%%%%%%%%%%%%%%%%%%%%%%%%%%%%%%%%%%%%%%%%%%

\subsection{Fourier Estimates}

%%%%%%%%%%%%%%%%%%%%%%%%%%%%%%%%%%%%%%%%%%%%%%%%

Now we present three results that have the same focus, as
they all bound convolution operators with a kernel such that the support of the Fourier transform is bounded away from $0$.  
The bounds are established in weighted H\"older norms and are the backbone of the proofs of the exchange lemmas.
 
In the first one, Lemma \ref{lem:guidos}, we consider some smooth projection on a region in Fourier space that is far away from the origin.
Using H\"older regularity, we show that this is an operator with small norm, when considered from $C^{0,\alpha}_\kappa$ to $C^0_\kappa$.
Later in Lemma \ref{lem:guidoext}  we modify the proof to give bounds on convolution operators using the $H^2$-norm of the Fourier transform of the kernel.
While in Corollary \ref{cor:new} we finally modify the result even more, by showing that we do not need the $L^2$-norm of the Fourier transform of the kernel. 
While Lemma \ref{lem:guidoext} will be sufficient for most of the estimates used in the proofs of the exchange lemmas, at one occasion we need Corollary \ref{cor:new}.

\begin{lemma}\label{lem:guidos}
	Let $\widehat{P}:\R\to [0,1]$ be a smooth function with bounded support, such that $0\notin\textup{supp}(\widehat{P})$. Let also $D\in \mathcal{C}^{0,\alpha}_\kappa$, with $\alpha\in (0,1)$, $\kappa>0$. 
	Then 
	\[
		\|P\ast D(\varepsilon\cdot)\|_{C^0_\kappa}\leqslant  C\varepsilon^\alpha\|D\|_{C^{0,\alpha}_\kappa}.
	\]
\end{lemma}

\begin{proof}
	Let us define $\widehat{G}=1-\widehat{P}$. Then, by taking the inverse Fourier transform we have for $x\in \R$
	\[
		P(x)+G(x)=\sqrt{2\pi}\delta_0. 
	\]
	Note also that $\widehat{G}(0)=1$.
	
	Now 
	\begin{equation*}
	\begin{split}
		P\ast D(\varepsilon\cdot) & = \sqrt{2\pi}D(\varepsilon\cdot) - G\ast D(\varepsilon\cdot)\\
		& = D(\varepsilon\cdot)\int_{\R}G(z)dz - \int_{\R}G(z)D(\varepsilon(z-\cdot))dz\\
		& = \varepsilon^\alpha\int_{\R}G(z)|z|^\alpha\frac{D(\varepsilon\cdot)-D(\varepsilon(z-\cdot))}{\varepsilon^\alpha|z|^\alpha}dz.
	\end{split}
	\end{equation*}
	Let us consider
	\[
		\begin{split}
			\|P\ast D(\varepsilon\cdot)\|_{C^0[-L,L]} & \leqslant  \varepsilon^\alpha\int_{\R}|G(z)||z|^\alpha\|D\|_{C^{0,\alpha}[-L_z,L_z]}dz\\
			& \leqslant  \varepsilon^\alpha\int_{\R}|G(z)||z|^\alpha L_z^\kappa\|D\|_{C^{0,\alpha}_\kappa}dz,
		\end{split}
	\]
	where $L_z=\max\{\varepsilon L+\varepsilon |z|,2\}$. We can now divide both sides by $L^\kappa$ to obtain
	\begin{equation}\label{e:so}
		\|P\ast D(\varepsilon\cdot)\|_{C^0_\kappa}\leqslant  \varepsilon^\alpha\int_{\R}|G(z)||z|^\alpha \Big(\frac{L_z}{L}\Big)^\kappa dz \|D\|_{C^{0,\alpha}_\kappa}.
	\end{equation}
	Now recall $\varepsilon\in (0,1)$ and $L>1$, so we derive
	\[
		\frac{L_z}{L}=\max\Big\{\varepsilon + \varepsilon\frac{|z|}{L},\frac{2}{L}\Big\}\leqslant  2+|z| \;.
	\]
	Going  back to \eqref{e:so}
	\[
		\|P\ast D(\varepsilon\cdot)\|_{C^0_\kappa}\leqslant  \varepsilon^\alpha\int_{\R}|G(z)||z|^\alpha (2+|z|)^\kappa dz \|D\|_{C^{0,\alpha}_\kappa}.
	\]
	The integral is actually finite and bounded by a constant $C_{\alpha,\kappa}$, as $\widehat{G}$ is sufficiently smooth. 
	Since any derivative has bounded support so $G$ decays sufficiently fast for the existence of the integral.
\end{proof}

A simple modification of the previous proof yields the following Lemma:

\begin{lemma}
 \label{lem:guidoext} 
	Let $\widehat{P}:\R\to [0,1]$ be a smooth function and $P$ its inverse Fourier transform. 
	Let also $D\in \mathcal{C}^{0,\alpha}_\kappa$, with $\alpha\in (0,1)$, $\kappa\in(0,1/2)$. 
	Then 
	\[
		\|P\ast D(\varepsilon\cdot)\|_{C^0_\kappa}\leqslant  C \varepsilon^\alpha \big[  \|\widehat{P}\|_{L^2}^{1-\alpha} \|\widehat{P}''\|_{L^2}^\alpha +   \|\widehat{P}''\|_{L^2} \big]\|D\|_{C^{0,\alpha}_\kappa} + |\widehat{P}(0)|\|D(\varepsilon\cdot)\|_{C^0_\kappa}\;.
	\]
\end{lemma}

\begin{proof}
 First note that 
 \[
[ P\ast D(\varepsilon\cdot)](x) = \int_{\R} P(y)  [D(\varepsilon (x-y)) - D(\varepsilon x) ] dy + \widehat{P}(0) D(\varepsilon x)\;.
 \]
 Then we can proceed as in the proof before to obtain 
 \begin{equation}
 \label{e:cent}
  \|P\ast D(\varepsilon\cdot)\|_{C^0_\kappa}
\leqslant  C\varepsilon^\alpha \int_{\R}|P(z)||z|^\alpha (2+|z|)^\kappa dz \|D\|_{C^{0,\alpha}_\kappa} + |\widehat{P}(0)|\|D(\varepsilon\cdot)\|_{C^0_\kappa}\;.
 \end{equation}
Now we bound the integral. For $|z|\geqslant1$ as $\alpha+\kappa \leqslant  3/2$
\[
\begin{split}
 \int_{\{|z|\geqslant1\}} |P(z)||z|^{\alpha+\kappa} dz 
 & \leqslant   \Big(\int_{\R}|P(z)|^2 |z|^4 dz \Big)^{\frac12} 
 \Big(\int_{\R} |z|^{2(\alpha+\kappa-2)} dz\Big)^{\frac12}\\
 & \leqslant  C \|\widehat{P}''\|_{L^2}\;,
\end{split}
\]
where we used Plancherel theorem in the last step.  For $|z|\leqslant  1$ we have 
\[
\begin{split}
 \int_{\{|z|\leqslant 1\}} |P(z)||z|^{\alpha} dz 
 & \leqslant    C \Big(\int_{\{|z|\leqslant 1\}} |P(z)|^2|z|^{2\alpha} dz\Big)^{1/2} \\
&\leqslant  C  \Big(\int_{\R}|P(z)|^2 dz \Big)^{(1-\alpha)/2} 
 \Big(\int_{\R} |P(z)|^2 |z|^{2} dz\Big)^{\alpha/2}\\
 & \leqslant  C \|\widehat{P}\|_{L^2}^{1-\alpha} \|\widehat{P}'\|_{L^2}^\alpha\;,
\end{split}
\]
where we used first the Cauchy-Schwarz inequality, then H\"older's one with $p=1/(1-\alpha)$ and $q=1/\alpha$, and finally Plancherel theorem.

Putting together all estimates finishes the proof.
\end{proof}

Unfortunately, the previous two lemmas are not sufficient in Lemma~\ref{lem:lemma1}.
When the support of $\widehat{P}$ is unbounded we have problems with the $L^2$-norm of $\widehat{P}$, 
while higher derivatives are easier to bound. 

\begin{corollary}
 \label{cor:new}
 Let $\widehat{P}:\R\to [0,1]$ be a smooth function, $P$ its inverse Fourier transform, 
 and suppose that there is some $\delta>0$ such that $\text{supp}(\widehat{P}) \cap (-\delta,\delta)=\emptyset$.   
 Fix $\alpha\in (0,1)$, $\kappa\in(0,1/2)$ and suppose that there is a $\gamma>1$ such that   $\alpha+\kappa+\gamma/2 \in(1,2)$. 
 
Then for all  $D\in \mathcal{C}^{0,\alpha}_\kappa$ we have
	\[
		\|P\ast D(\varepsilon\cdot)\|_{C^0_\kappa}
		\leqslant  C \varepsilon^\alpha  \|\widehat{P}'\|_{L^2}^{2-\alpha-\kappa-\frac{\gamma}2}  \|\widehat{P}''\|_{L^2}^{\alpha+\kappa+\frac{\gamma}2-1 } \|D\|_{C^{0,\alpha}_\kappa}\;.
	\]
\end{corollary}

\begin{proof}
	From~\eqref{e:cent} we obtain using H\"older's inequality
\[
\begin{split}  
\|P\ast D(\varepsilon\cdot)\|_{C^0_\kappa}
& \leqslant  C\varepsilon^\alpha \int_{\R}|P(z)||z|^\alpha (2+|z|)^\kappa dz \|D\|_{C^{0,\alpha}_\kappa} \\
&\leqslant  C\varepsilon^\alpha\int_{\{|z|\geqslant \delta\}} |P(z)||z|^{\alpha+\kappa} dz\|D\|_{C^{0,\alpha}_\kappa} \\
& \leqslant  C\varepsilon^\alpha  \Big(\int_{\{|z|\geqslant \delta\}} |P(z)|^2|z|^{2\alpha+2\kappa+\gamma} dz\Big)^{1/2} \|D\|_{C^{0,\alpha}_\kappa}\\
& \leqslant  C\varepsilon^\alpha  \Big(\int_{\mathbb{R}} |zP(z)|^2|z|^{2\alpha+2\kappa+\gamma-2} dz\Big)^{1/2} \|D\|_{C^{0,\alpha}_\kappa}.
\end{split}
\]
Now as the exponent $2\alpha+2\kappa+\gamma-2\in(0,2)$, 
we can use H\"older inequality to bound the integral above by the integral over  $|z|^2|P(z)|^2$ and  $|z|^4|P(z)|^2$, 
which in turn gives the $L^2$-norm of $\widehat{P}'$ and $\widehat{P}''$.
We obtain
\[\begin{split}  
\int_{\mathbb{R}} |zP(z)|^2|z|^{2\alpha+2\kappa+\gamma-2} dz
&\leqslant \Big( \int_{\mathbb{R}}   |zP(z)|^2 dz  \Big)^{2-\alpha-\kappa-\frac{\gamma}2} 
\Big( \int_{\mathbb{R}}    |z^2P(z)|^2 dz   \Big)^{\alpha+\kappa+\frac{\gamma}2-1 }\\
&\leqslant \|\widehat{P}'\|_{L^2}^{2(2-\alpha-\kappa-\frac{\gamma}2)}  \|\widehat{P}''\|_{L^2}^{2(\alpha+\kappa+\frac{\gamma}2-1) },
\end{split}\]
which implies the claim.
\end{proof}

%%%%%%%%%%%%%%%%%%%%%%%%%%%%%%%%%%%%%%%%%%%%%%%%

\subsection{Applications of Fourier estimates}

%%%%%%%%%%%%%%%%%%%%%%%%%%%%%%%%%%%%%%%%%%%%%%%%

Now we rephrase the bounds of the previous subsection to bound operators given by a Fourier multiplier, as for example
in the statement of the exchange Lemmas~\ref{lem:lemma1}, \ref{lem:lemma2}, and \ref{lem:lemma3}.  
Another  example we have in mind are bounds on the semigroup generated by the Swift-Hohenberg operator which is presented later in Corollary \ref{cor:SGbound}.

In the first step we use regularity of the kernel to bound the operator.
\begin{lemma}
\label{lem:ours}
	Let $m>\frac{1}{2}$ and $\mathcal{H}\cdot=H\star\cdot$ be an operator such that the Fourier transform 
	$\hat{H}$ of the kernel $H$ is in $H^m(\R)$.
	For any $\kappa \in (0,  m -\frac12)$ there is a constant such that  
	 for all $u\in C^0_\kappa$ we have  
	\[ 
	\|\mathcal{H}  u\|_{C^0_\kappa} \leqslant  C \|\hat{H}\|_{H^m(\R)} \| u\|_{C^0_\kappa} \;.
	\]
\end{lemma}

\begin{proof} By the definition of the convolution
	\begin{equation*}
	\begin{split}
	\mathcal{H}  u (x) &= \int_{\R} H(x-y) u(y) dy = C\int_{\R} H(z) u(x-z) dz \\
	&= \int_{\R} H(z) (1+(x-z)^2)^{\kappa/2}  w_\kappa(x-z) u(x-z) dz.
	\end{split}
	\end{equation*}
	Now we use
	\[1+(x-z)^2 \leqslant  2(1+z^2)(1+x^2)\]
	to obtain
	\begin{equation*}
\begin{split}
	|\mathcal{H}  u (x)|
	&= C(1+x^2)^{\kappa/2}  \int_{\R}|H(z)| (1+z^2)^{\kappa/2}  dz \| u\|_{C^0_\kappa}\\
	& \leqslant   C(1+x^2)^{\kappa/2}  \Big( \int_{\R}(1+z^2)^{-m+\kappa} dz\Big)^{\frac{1}{2}}    \Big( \int_{\R}|H(z)|^2 (1+z^2)^{m}  dz\Big)^{\frac{1}{2}}  \| u\|_{C^0_\kappa}.
\end{split}
	\end{equation*}
	We finish the proof by noting that $ m -\kappa >\frac12$ by assumption, and that by Plancherel theorem
	\[
	\int_{\R}|H(z)|^2 (1+z^2)^{m}  dz =\|\hat{H}\|_{H^m(\R)} \;.   \qedhere
	\]
\end{proof}
In order for the previous Lemma to be useful in our case, we have to control the $H^m$-norm of the kernel. 
This is straightforward for the Swift-Hohenberg semigroup 
if we add a smooth projection on bounded Fourier domains. 
\begin{lemma}\label{lem:thisone}
	Fix $m\in[0,1)$ and $\ell\in\mathbb{Z}$. Consider  $\widehat{P} :
	\mathbb{R} \rightarrow [0, 1]$ smooth with $\textup{supp} (\widehat{P}) \subset [\ell - 2 \delta, \ell + 2
	\delta]$ for some  $0<\delta < 1 / 2$. Then it holds that
	\[
		\sup_{t \in [0, T_0 \varepsilon^{-2}]} \|\widehat{P}e^{\lambda_\nu t}\|_{H^m}\leqslant  C \varepsilon^{-\max\{0,m-\frac12\}}
	\]
	where $\lambda_\nu(k)=-(1-k^2)^2+\nu\varepsilon^2= - (1 - k)^2  (1+ k)^2 + \nu \varepsilon^2$ is the Fourier symbol of the operator $\cL_\nu$.
\end{lemma}

\begin{corollary}
\label{cor:SGbound}
Consider the Fourier projection $\mathcal{P}=P\star\cdot$ with $\widehat{P}$ as in the lemma above,
then we obtain in case $\kappa \in (0, \frac14)$ with $m=\frac12+\kappa$ that 

	\[
		\|e^{\mathcal{L}_\nu t}\mathcal{P}u\|_{C^0_\kappa}\leqslant  C\|\widehat{P}e^{\lambda_\nu t}\|_{H^m}\|u\|_{C^0_\kappa}
		\leqslant  C \varepsilon^{-\kappa} \|u\|_{C^0_\kappa}\;.
	\]
	for all $t\in[0,T_0\varepsilon^{-2}]$.
\end{corollary}

\begin{proof}[Proof of Lemma~\ref{lem:thisone}]
For the proof we only focus on the most complicated case $\ell=1$, i.e.~with $\textup{supp}(P)\subset[1-2\delta,1+2\delta]$.  The  
case $\ell=-1$ is almost verbatim and for $|\ell|\not=1$ the proof is actually much simpler, 
as  $\lambda_\nu \leqslant -c <0$ is strictly negative there and we obtain exponentially small terms.

 A straightforward calculation shows
	\begin{equation*}
	\begin{split}
		\| \widehat{P} e^{\lambda_\nu t} \|^2_{H^1} & =  \| \widehat{P} (\cdot - 1)
		e^{\lambda_\nu (\cdot - 1) t} \|^2_{H^1 (- 2 \delta, 2 \delta)}\\
		& \leqslant C
		\int_{- 2 \delta}^{2 \delta} e^{- 2 (2 - k)^2 k^2 t} d k \\
		& \qquad + C \int_{- 2 \delta}^{2 \delta} 4 (2 (k - 2) k^2 + (k - 2)^2 k)^2 t^2
		e^{- 2 (2 - k)^2 k^2 t} d k\\
		& \leqslant C \int_{- 2 \delta}^{2 \delta}
		e^{- C_{\delta} k^2 t} d k + C \int_{- 2 \delta}^{2 \delta} (k^2 + k)^2
		t^2 e^{- C_{\delta} k^2 t} d k\\
		& \leqslant  C + C \int_{- 2 \delta}^{2 \delta} k^2
		t^2 e^{- C_{\delta} k^2 t} d k.
	\end{split}
	\end{equation*}
	Now we have to consider two cases, depending on $t$. First if $t \leqslant 1$, then
	also the second integral can be bound by a constant $C$. If $t > 1$, then we can
	continue with the  substitution  $l = \sqrt{t} k$, which gives
	$d l = \sqrt{t} d k$, and we derive
	\begin{equation*}
		\begin{split}
		\| \widehat{P} e^{\lambda_\nu t} \|^2_{H^1} & \leqslant  C + C \int_{- 2
			\delta \sqrt{t}}^{2 \delta \sqrt{t}} l^2  \sqrt{t} e^{- Cl^2} d l
	   \leqslant  C + C   \sqrt{t} \int_{- \infty}^{\infty} l^2  e^{- Cl^2} d l\\
		& \leqslant  C \sqrt{t} \;.
		\end{split}
	\end{equation*}
	Thus
	\[ \| \widehat{P}  e^{\lambda_\nu t} \|_{H^1}^2 \leqslant \left\{
	\begin{array}{ll}
	C \sqrt{t} \;, & t \geqslant 1\\
	C \;, & t \leqslant 1
	\end{array} \right. \;.\]
	In a similar way, we can consider the bounds in $L^2$:
	\begin{equation*}
		\| \widehat{P}  e^{\lambda_\nu  t} \|^2_{L^2} \leqslant C \int_{- 2
			\delta}^{2 \delta} e^{- C_{\delta} k^2 t} d k 
		 = C \int_{- 2 \delta \sqrt{t}}^{2 \delta \sqrt{t}} \frac{1}{\sqrt{t}}
		e^{- C_{\delta} l^2} d l 
		 \leqslant \left\{
	\begin{array}{ll}
	C t^{-1/2} \;, & t \geqslant 1\\
	C  \;,& t \leqslant 1
	\end{array} \right. 
		\;.
	\end{equation*}
	We finally get to the bounds in $H^m$ by interpolation:
	\begin{equation*}
		\| \widehat{P}  e^{\lambda_\nu  t} \|_{H^m} 
		 \leqslant  C \| \widehat{P}e^{\lambda_\nu  t} \|^{1 - m}_{L^2}  \| \widehat{P} e^{\lambda_\nu t} \|^m_{H^1}
		 \leqslant  \left\{ \begin{array}{ll}
			C\;,&  t \leqslant 1\\
			Ct^{- \frac{1}{4}+\frac{m}{2}}\;,&   t \geqslant 1
		\end{array} \right\} \leq C \varepsilon^{-\max\{0,m-\frac12\}}
	\end{equation*}
	for all $t\leqslant T_0\varepsilon^{-2}$.
\end{proof}

%%%%%%%%%%%%%%%%%%%%%%%%%%%%%%%%%%%%%%%%%%%%%%%%%%%%%%%%%%%%%%%%5

\subsection{Proof of Exchange Lemma II}

%%%%%%%%%%%%%%%%%%%%%%%%%%%%%%%%%%%%%%%%%%%%%%%%

For the proof of Lemma~\ref{lem:lemma3} we write the differences of semigroups as convolution operators.

First we define a smooth Fourier multiplier that cuts out regions around $\pm1$ in Fourier space, 
where the eigenvalues of the Swift-Hohenberg operator are close to $0$.
Fix a small $\delta>0$ independent of $0<\varepsilon\ll 1$ and consider a smooth function $\widehat{P}:\R\to[0,1]$ such that  
$\textup{supp}(\widehat{P})=[-1-2\delta, -1+2\delta]\cup[1-2\delta, 1+2\delta]$
and $\widehat{P}=1$ on $[-1-\delta, -1+\delta]\cup[1-\delta, 1+\delta]$.
We define  $\hat{Q}=1-\widehat{P}^2$ and let $\mathcal{Q}=I-\mathcal{P}^2$.

Now we obtain
\[ 
e^{T \varepsilon^2 \mathcal{L}_{\nu}} [D (\varepsilon \cdot) \mathrm{e}_3] 
=
e^{T \varepsilon^2 \mathcal{L}_{\nu}} \mathcal{P} [\mathcal{P}D
(\varepsilon \cdot) \mathrm{e}_3] + e^{T \varepsilon^2 \mathcal{L}_{\nu}}
\mathcal{Q} [D (\varepsilon \cdot) \mathrm{e}_3]
\]
and we bound separately the two terms. For the first term we use the semigroup
estimate from Corollary \ref{cor:SGbound} (with $\ell=\pm1$), 
using the $H^\alpha$-estimate on the kernel and Lemma \ref{lem:guidos}.
Note that for the application we need to split the estimate
into two terms: one concentrated around $1$ and the other around $-1$.
We obtain
\[ 
\| e^{T \varepsilon^2 \mathcal{L}_{\nu}} \mathcal{P} [\mathcal{P}D
(\varepsilon \cdot) \mathrm{e}_3] \|_{C_{\kappa}^0}
\leqslant C \varepsilon^{-\kappa}
\| \mathcal{P}D
(\varepsilon \cdot) \mathrm{e}_3 \|_{C_{\kappa}^0}
\leqslant C
\varepsilon^{\alpha-\kappa}  \| D \|_{C_{\kappa}^{0, \alpha}}\;.
\]
For the second term we need some more work. We start by writing
\[ 
e^{T \varepsilon^2 \mathcal{L}_{\nu}} \mathcal{Q} [D (\varepsilon \cdot)
\mathrm{e}_3] = (\mathcal{H}_T D) (\varepsilon \cdot) \mathrm{e}_3
\]
and denoting the kernel of $\mathcal{H}_T=H_T\star$ by
\[ 
{\hat{H}_T} (k) = e^{\nu T} e^{- T \varepsilon^{- 2}  (4 + k \varepsilon)^2  (2 + k
	\varepsilon)^2}  \hat{Q} (3 + \varepsilon k) \;.
\]
In view of Lemma \ref{lem:ours} we only need to bound the $H^1$-norm of the kernel ${\hat{H}_T}$.
Therefore, we  split the $H^1$-norm into two different areas in Fourier space
\[ \| {\hat{H}_T} \|^2_{H^1 (\mathbb{R})} \leqslant 2 \| {\hat{H}_T} \|^2_{H^1 \left( \left[ -
	\frac{c}{\varepsilon}, \frac{c}{\varepsilon} \right] \right)} + 2 \| {\hat{H}_T}
\|^2_{H^1 \left( \left[ - \frac{c}{\varepsilon}, \frac{c}{\varepsilon}
	\right]^C \right)}\;.
	\]
Note first that both $\hat{Q}$ and $\hat{Q}'$ are bounded smooth functions independent of $\varepsilon$.
Then we use in the first term that $\varepsilon | k | \leqslant C$ and that $| T | $ is bounded.

Thus ${\hat{H}_T}$ is uniformly bounded on $[ -\frac{c}{\varepsilon}, \frac{c}{\varepsilon}]$ by $ C e^{- TC_0 \varepsilon^{- 2}}$ 
where $- C_0$ is the level
where we cut out the two bumps around $- 2 / \varepsilon$ and $- 4 /\varepsilon$. Note that $\hat{Q}$  is identically zero there.
With similar arguments we show that the derivative ${\hat{H}_T}'$ is uniformly bounded by
$ C \varepsilon (1+T\varepsilon^{-2}) e^{- TC_0 \varepsilon^{- 2}}$. Thus
\begin{equation*}
	\begin{split}
		\| {\hat{H}_T} \|^2_{H^1 \left( \left[ - \frac{c}{\varepsilon}, \frac{c}{\varepsilon}
		\right] \right)}
		& \leqslant  C \int_{-\frac{c}{\varepsilon}}^{\frac{c}{\varepsilon}} (1 + \varepsilon^2+ T^2 \varepsilon^{-2}) e^{- 2TC_0 \varepsilon^{- 2}} dk\\
		& \leqslant  C \varepsilon^{- 1}  (1 + T^2 \varepsilon^{- 2}) e^{-2 TC_0
		\varepsilon^{- 2}}\\
		& \leqslant  C \varepsilon^{- 1}  e^{- TC_0 \varepsilon^{- 2}}\\
		& \leqslant  C \varepsilon^{4 \gamma - 1} T^{- 2 \gamma}\;,
	\end{split}
\end{equation*}
where we used first that $xe^{-x} \leqslant  1$ and then that $e^{-x} \leqslant  C_\gamma x^{-\gamma}$.
The final estimate is not necessary at this point, but it is still sufficient for our purposes, 
as other terms in the estimate are bounded by this weaker estimate.

Now we have to consider the case $\varepsilon | k | > c$ 
when we are away from the bumps.  In this case, by adjusting $c$ we can use that $\hat{Q}$ is a constant.
Moreover, the bound for negative and positive $k$ is the same, so we restrict ourselves to the case with $k>c/\varepsilon$. 
\begin{equation*}
	\begin{split}
		\| {\hat{H}_T} \|^2_{H^1 \left( \left[ \frac{c}{\varepsilon}, + \infty \right)
		\right)} & \leqslant \int_{\frac{c}{\varepsilon}}^{\infty} e^{2 T
		\varepsilon^{- 2}  (4 + k \varepsilon)^2  (2 + k \varepsilon)^2} dk \\
		&\quad + \int_{\frac{c}{\varepsilon}}^{\infty} e^{2 T \varepsilon^{- 2}  (4 +
		k \varepsilon)^2  (2 + k \varepsilon)^2}  (T \varepsilon^{- 2}  (4 + \varepsilon k)  (2+ \varepsilon k)  (3 +\varepsilon k) \varepsilon)^2 dk\\
	& \leqslant \frac{1}{\varepsilon}  \int_c^{\infty} e^{2
		T \varepsilon^{- 2}  (4 + k)^2  (2 + k )^2} dk \\
	& \quad + \frac{C}{\varepsilon}  \int_c^{\infty} e^{2 T \varepsilon^{- 2}  (4
		+ k )^2  (2 + k )^2}  (T \varepsilon^{- 1}  (4 + k) 
	(2 + k)  (3 + k))^2 dk \;.
	\end{split}
\end{equation*}
Now we use that $(4 + k)^2  (2 + k)^2 \geqslant k^4$ and $(4 + k)  (2 + k)  (3
+ k) \leqslant C (k^3 + 1) \leqslant Ck^3$, for $|h|> \frac{C}{\varepsilon}$ with an $\varepsilon$-independent constant $C$, so that
\begin{equation*}
	\begin{split}
		\| {\hat{H}_T} \|^2_{H^1 \left( \left[ \frac{c}{\varepsilon}, + \infty \right)
		\right)} & \leqslant  \frac{C}{\varepsilon}  \int_c^{\infty} e^{- CT
		\varepsilon^{- 2} k^4} dk + CT^2 \varepsilon^{- 3}  \int_c^{\infty} k^3 e^{-
		CT \varepsilon^{- 2} k^4} dk\\
		& =
		\frac{C}{\varepsilon}  (T \varepsilon^{- 2})^{- \frac{1}{4}}  \int_{c (T
		\varepsilon^{- 2})^{\frac{1}{4}}}^{\infty} e^{- Ck^4} dk \\
		& \quad+ CT^2 \varepsilon^{- 3}  (T \varepsilon^{- 2})^{- \frac{1}{4}} 
		\int_{c (T \varepsilon^{- 2})^{\frac{1}{4}}}^{\infty} k^3  (T \varepsilon^{-
		2})^{- \frac{3}{4}} e^{- Ck^4} dk\\
		& \leqslant \frac{C}{\varepsilon}  (T \varepsilon^{- 2})^{- \frac{1}{4}} 
		\int_{c (T \varepsilon^{- 2})^{\frac{1}{4}}}^{\infty} e^{- Ck^4} dk + T
		\varepsilon \int_{0}^{\infty} k^3 e^{-
		Ck^4} dk \\
		& \leqslant  \frac{C}{\varepsilon}  (T \varepsilon^{- 2})^{- \frac{1}{4}} 
	\int_{c (T \varepsilon^{- 2})^{\frac{1}{4}}}^{\infty} e^{- Ck^4} dk + C\varepsilon \;.
	\end{split}
\end{equation*}
For the remaining term we use that for $\alpha \geqslant 0$ 
\[ 
	\sup_{z>0}\Big\{z^{\alpha}  \int_z^{\infty} e^{- ck^4} dk\Big\} <\infty,
\]
to obtain for $\gamma= (1+\alpha)/8 \geqslant 1/8 $ 
\[
\| {\hat{H}_T} \|^2_{H^1 \left( \left[ \frac{c}{\varepsilon}, + \infty \right)
		\right)} 
		 \leqslant  C \varepsilon^{4\gamma - 1} T^{- 2 \gamma}+   C\varepsilon \;.
\]
Note finally,  that for $\gamma<1/2$ and bounded $T$ 
we have $\varepsilon \leqslant  C \varepsilon^{4\gamma - 1} T^{- 2 \gamma}$ and 
we can neglect the $ C\varepsilon$ in the estimate above.

%%%%%%%%%%%%%%%%%%%%%%%%%%%

\subsection{Proof of Exchange Lemma I}

%%%%%%%%%%%%%%%%%%%%%%%%%%%%%%%%%%%%%%%%%%%%%%%%

The proof of the Exchange Lemma I stated in Lemma~\ref{lem:lemma2} is similar to the one for the Exchange Lemma II in Lemma~\ref{lem:lemma3}, but requires additional arguments.
We start again by smoothly projecting in Fourier space, but now in $k=1$ and $k=3$.

Fix a small $\delta>0$ and consider for $\ell\in\mathbb{Z}$ a smooth function $\hat{P_\ell}:\R\to[0,1]$ such that  
$\textup{supp}(\hat{P_\ell})=[-\ell-2\delta, -\ell+2\delta]\cup[\ell-2\delta, \ell+2\delta]$
and $\widehat{P}_\ell=1$ on $[-\ell-\delta, -\ell+\delta]\cup[\ell-\delta, \ell+\delta]$.

Now we can rewrite:
\begin{multline*}
	e^{t\cL_\nu}[D(\varepsilon \cdot)\e_1] - (e^{\Delta_\nu T}D)(\varepsilon \cdot)\e_1  = \mathcal{P}_3^2 e^{t\cL_\nu}[D(\varepsilon \cdot)\e_1] + \mathcal{P}_1^2 e^{t\cL_\nu}[D(\varepsilon \cdot)\e_1]\\ 
	- \mathcal{P}_1^2 (e^{\Delta_\nu T}D)(\varepsilon \cdot)\e_1 +(1-\mathcal{P}_1^2-\mathcal{P}_3^2)e^{t\cL_\nu}[D(\varepsilon \cdot)\e_1] -(1-\mathcal{P}_1^2)(e^{\Delta_\nu T}D)(\varepsilon \cdot)\e_1.
\end{multline*}
Now the first term on the right hand side is bounded the same way as the second term in the proof of the Exchange Lemma II (Lemma~\ref{lem:lemma3}) in the previous section.
Also the last two terms can be controlled in a similar way as the first term in the proof of Lemma~\ref{lem:lemma3}.
We only need the semigroup
estimate from Corollary \ref{cor:SGbound} (now for $\ell=\pm1$ and $\ell=\pm3$) and Lemma \ref{lem:guidos}.

Let us focus on the missing two terms:
\[
	\mathcal{P}_1^2 e^{t\cL_\nu}[D(\varepsilon \cdot)\e_1] - \mathcal{P}_1^2 (e^{\Delta_\nu T}D)(\varepsilon \cdot)\e_1 =: \mathcal{H}D,
\]
with $\mathcal{H}\cdot=H\star\cdot$ such that $\textup{supp}(\hat{H})\subset (-2\delta/\varepsilon, 2\delta/\varepsilon)$ and 
\[
\begin{split}
	\hat{H} & = \widehat{P}(\varepsilon\cdot)[e^{T\varepsilon^{-2}\lambda_\nu(1+l\varepsilon)}-e^{-4l^2T+\nu T}]\\
	& = \widehat{P}(\varepsilon\cdot)e^{-4Tl^{2}+\nu T}[e^{-4Tl^{3}\varepsilon-l^4T\varepsilon^2}-1].
\end{split}
\]
In view of Lemma \ref{lem:ours} it is enough to show that $\|\hat{H}\|_{H^1(\R)}$ is small. Thus we need the derivative
\begin{multline*}
\frac{d}{dl}\hat{H} = \widehat{P}^\prime(\varepsilon\cdot)e^{-4Tl^{2}+\nu T}[e^{-4Tl^{3}\varepsilon-l^4T\varepsilon^2}-1]-8Tl \widehat{P}(\varepsilon\cdot)e^{-4Tl^{2}+\nu T}[e^{-4Tl^{3}\varepsilon-l^4T\varepsilon^2}-1]\\
-4T\varepsilon l^2(3+4\varepsilon l) \widehat{P}(\varepsilon\cdot)e^{-4Tl^{2}+\nu T}[e^{-4Tl^{3}\varepsilon-l^4T\varepsilon^2}-1].
\end{multline*}
Now we collect the common exponential term in the parenthesis, and then we write the Taylor expansion in $l=0$.
 To get the estimate in $H^1$ we bound both $\hat{H}$ and $\frac{d}{dl}\hat{H}$ in $L^2$.
 
Actually we can do better and provide pointwise estimates (and not just $L^2$).  First of all we observe that by means of Taylor expansion and triangular inequality
\[
  |1-e^z|\leqslant  |z|e^{|z|},
\]
and in our case $z = -4Tl^{3}\varepsilon-l^4T\varepsilon^2$, with $|z|\leqslant  8T\delta l^2 + 4\delta^2Tl^2\leqslant  2Tl^2$ 
if $|\ell| \leqslant  \delta/\varepsilon$ for some small fixed $\delta\leqslant  1/2$.

If we consider $\hat{H}$ we have the following bound:
\begin{equation*}
	\begin{split}
		|\hat{H}| & \leqslant  C_{T_0}e^{-2Tl^2}\left(|4Tl^3\varepsilon|+ |l^4T\varepsilon^2|\right)\\
		& \leqslant  C_{T_0}T^{-1/2}e^{-Tl^2}\left(4\varepsilon+ l\varepsilon^2\right)\\
		& \leqslant  C_{T_0}T^{-1/2}e^{-Tl^2}\left(4+2\delta\right)\varepsilon\;,
	\end{split}
\end{equation*}
where we used the inequality
\[
	e^{-Tl^2}\left(Tl^2\right)^{3/2}\leqslant  C.
\]

In the same way we can bound the derivative of $\hat{H}$:
\[
	\hat{H}^\prime = \left(\varepsilon\widehat{P}^\prime - 8Tl\widehat{P}-4T\varepsilon l^2(3+4\varepsilon l)\widehat{P}\right)e^{-4Tl^{2}+\nu T}[e^{-4Tl^{3}\varepsilon-l^4T\varepsilon^2}-1],
\]
where we have almost $\hat{H}$ with a different prefactor that we can bound by using the previous one on $\hat{H}$ and the fact that $\varepsilon l\leqslant  \delta$ such that
\[
	|\hat{H}^\prime|\leqslant  C(1+Tl)T^{-1/2}e^{-Tl^2}\varepsilon.
\]
Now we use that $\sqrt{T}le^{-1/2\cdot Tl^2}\leqslant  C$, so
\[
	|\hat{H}^\prime|\leqslant  C(1+\sqrt{T})T^{-1/2}e^{-1/2\cdot Tl^2}\varepsilon\;.
\]
Thus using $T\leqslant  T_0$
\[
	\begin{split}
		\|\hat{H}\|_{H^1} & \leqslant  CT^{-1/2}\varepsilon\|^{-1/2\cdot Tl^2}\|_{L^2}\\
		& \leqslant  C T^{-3/4}\varepsilon \;.
	\end{split}
\]

%%%%%%%%%%%%%%%%%%%%%%%%%%%%%%%%%%%%%%%%%%%%%%%%

\subsection{Proof of Exchange Lemma IC}

%%%%%%%%%%%%%%%%%%%%%%%%%%%%%%%%%%%%%%%%%%%%%%%%

The idea behind this proof is almost the same as before, but the proof itself is technically slightly different, relies on Corollary~\ref{cor:new}, and does not need $L^2$-estimates on the kernel.

We start again by smoothly projecting in Fourier space, but onto the modes  $k=\pm1$ and $k=\pm3$.

Fix a small $\delta>0$ and consider  a smooth function $\widehat{P}:\R\to[0,1]$ such that  
$\textup{supp}(\widehat{P})=[-2\delta, 2\delta]$
and $\widehat{P}=1$ on $[-\delta, \delta]$.
Define now for $\ell\in\mathbb{Z}$ the function 
\[\widehat{P}_\ell:\R\to[0,1] \text{ defined by }\widehat{P}_\ell(k)=\widehat{P}(k-\ell).\]
Now we can rewrite:
\begin{equation}
\label{e:ELIC}
\begin{split}
	e^{t\cL_\nu}[D(\varepsilon \cdot)\e_1] - (e^{\Delta_\nu T}D)(\varepsilon \cdot)\e_1 
	 = & \mathcal{P}_3^2 e^{t\cL_\nu}[D(\varepsilon \cdot)\e_1] \\
	& + \mathcal{P}_1 e^{t\cL_\nu}[D(\varepsilon \cdot)\e_1] - \mathcal{P}_1 (e^{\Delta_\nu T}D)(\varepsilon \cdot)\e_1 \\
	& +(1-\mathcal{P}_1-\mathcal{P}_3^2)e^{t\cL_\nu}[D(\varepsilon \cdot)\e_1] \\ 
	&- (1-\mathcal{P}_1)(e^{\Delta_\nu T}D)(\varepsilon \cdot)\e_1.
\end{split}
\end{equation}
%
%
%%%%%%%%%%%%%%%%%%%%%%%%%%%%%%%%%%%%%%%%%%%%%%%%
%
\subsubsection{First term}
%
%%%%%%%%%%%%%%%%%%%%%%%%%%%%%%%%%%%%%%%%%%%%%%%%
%
%
Now the first term is bounded the same way as the second term in the proof the Exchange Lemma II (Lemma~\ref{lem:lemma3}).
We only need the semigroup
estimate from Corollary \ref{cor:SGbound} and Lemma \ref{lem:guidos} to obtain.
\[ 
\| e^{T \varepsilon^2 \mathcal{L}_{\nu}} \mathcal{P}_3^2 [D(\varepsilon \cdot) \mathrm{e}_1] \|_{C_{\kappa}^0}
\leqslant C
\varepsilon^{\alpha-\kappa}  \| D \|_{C_{\kappa}^{0, \alpha}}\;.
\]
%
%
%%%%%%%%%%%%%%%%%%%%%%%%%%%%%%%%%%%%%%%%%%%%%%%%
%
\subsubsection{Second term}
%
%%%%%%%%%%%%%%%%%%%%%%%%%%%%%%%%%%%%%%%%%%%%%%%%
%
We can write the second term in view of Lemma \ref{lem:guidoext}:
\[
 \mathcal{P}_1 e^{t\cL_\nu}[D(\varepsilon \cdot)\e_1] - \mathcal{P}_1 (e^{\Delta_\nu T}D)(\varepsilon \cdot)\e_1 =  \mathcal{H} [D(\varepsilon \cdot)]  \cdot e_1
\]
with convolution operator $ \mathcal{H}_T\cdot = H_T\star\cdot$ with Fourier transform
\[
\begin{split}
 \hat{H}_T(k)& = \widehat{P}(k)[ e^{-T \varepsilon^{-2}k^2( k+2)^2} -  e^{-4k^2 \varepsilon^{-2} T}]e^{\nu T}
 \\& = \widehat{P}(k)[ e^{-T\varepsilon^{-2}(4k^3+k^4)} -  1]e^{-4k^2 \varepsilon^{-2} T}e^{\nu T},
 \end{split}
\]
where  $\hat{H}_T(0)=0$.
Now we bound the $L^2$-norms  of $\hat{H}_T$, $\hat{H}^\prime_T$, and $\hat{H}''_T$, and apply the results in Lemma~\ref{lem:guidoext}.
We can get the following pointwise bound, 
using the support of $\widehat{P}$ together with mean value theorem and $\delta<1/2$
\begin{equation}
\label{e:*}
\begin{split}
 |\hat{H}_T(k)|& \leqslant  C |\widehat{P}(k) T\varepsilon^{-2}|4k^3+k^4| e^{-k^2 \varepsilon^{-2} T}
 \\&  \leqslant  C |\widehat{P}(k)| T\varepsilon^{-2}|k|^3 e^{-k^2 \varepsilon^{-2} T},
 \end{split}
\end{equation}
for all $\gamma \geqslant0$.

We use that for $a>0$ and $\xi>0$
\[
\begin{split}
\int_0^{2\delta}  k^a e^{-\xi k^2} dk  & 
  =  \int_0^{2\delta\sqrt{\xi}} \xi^{-1/2-a/2} k^a e^{- k^2} dk 
 \\&  \leqslant  C \min\{  \xi^{-1-a}\ , \ 1 \}^{1/2}.
 \end{split}
\]
Thus for the $L^2$-norm we integrate the squared inequality \eqref{e:*} 
and use the previous estimate with $a=6$ and $\xi=T\varepsilon^{-2}$ to obtain

\[
\|\hat{H}_T\|_{L^2} \leqslant  C (T\varepsilon^{-2})\min\{ (T\varepsilon^{-2})^{-7}\ , \ 1 \}^{1/4} \leqslant  C \;.
\]

For the first derivative
\[
\begin{split}
 \hat{H}'_T (k)
  =&  \widehat{P}'(k)[ e^{-4T\varepsilon^{-2}(4k^3+k^4)} -  1]e^{-4k^2 \varepsilon^{-2} T}e^{\nu T} \\
 & +  \widehat{P}(k)[ (-T\varepsilon^{-2}(12k^2+4k^3)) e^{-T \varepsilon^{-2}(4k^3+k^4)} e^{-4k^2 \varepsilon^{-2} T}e^{\nu T} \\
  & + \widehat{P}(k)[ e^{-T\varepsilon^{-2}(4k^3+k^4)} -  1] (-8k \varepsilon^{-2} T) e^{-4k^2 \varepsilon^{-2} T}e^{\nu T}.
 \end{split}
\]
As before, 
\[
\begin{split}
 |\hat{H}'_T (k)|
  \leqslant  &  C |\widehat{P}'(k)|  T\varepsilon^{-2} |k|^3 e^{-4k^2 \varepsilon^{-2} T} 
  + C |\widehat{P}(k)| [ (T\varepsilon^{-2}) k^2 +  (T\varepsilon^{-2})^2 k^4 ]e^{-k^2 \varepsilon^{-2} T}   \\
    \leqslant  &  C_\delta |\widehat{P}'(k)|  T\varepsilon^{-2}  e^{- \delta^2  \varepsilon^{-2} T} 
  + C |\widehat{P}(k)|  (T\varepsilon^{-2}) k^2 e^{-k^2 \varepsilon^{-2} T} . 
 \end{split}
\]
Thus for the $L^2$-norm
\[
\|\hat{H}'_T\|_{L^2} \leqslant  C+ C (T\varepsilon^{-2})\min\{ (T\varepsilon^{-2})^{-5}\ ; \ 1 \}^{1/4} \leqslant  C \;.
\]

For the second derivative we obtain similarly
\[
\begin{split}
 |\hat{H}''_T (k)|
  \leqslant  &  
   C_\delta |\widehat{P}'' (k)|  T\varepsilon^{-2}  e^{- \delta^2  \varepsilon^{-2} T} \\
  &  +  C_\delta |\widehat{P}'(k)| [T\varepsilon^{-2} + (T\varepsilon^{-2})^2]  e^{- \delta  \varepsilon^{-2} T} \\
 &+ C |\widehat{P}(k)| [ (T\varepsilon^{-2}) |k|  + (T\varepsilon^{-2})^2 |k|^3 +  (T\varepsilon^{-2})^3 |k|^5]   e^{-k^2 \varepsilon^{-2} T} ,
 \end{split}
\]
so for the $L^2$-norm
\[
\|\hat{H}''_T\|_{L^2} \leqslant  C+ C (T\varepsilon^{-2})\min\{ (T\varepsilon^{-2})^{-3}\ , \ 1 \}^{1/4} \leqslant  C \varepsilon^{-1/2} \;.
\]
Now Lemma \ref{lem:guidoext} yields: 
 \[
\|  \mathcal{P}_1 e^{t\cL_\nu}[D(\varepsilon \cdot)\e_1] - \mathcal{P}_1 (e^{\Delta_\nu T}D)(\varepsilon \cdot)\e_1\|_{C^0_\kappa}
\leqslant  C\varepsilon^\alpha (1+\varepsilon^{-1/2}) \|D\|_{C^{0,\alpha}_\kappa}.
 \]
%
 %  %%%%%%%%%%%%%%%%%%%%%%%%%%%%%%%%%%%%%%%%%%%%%%
 %
 \subsubsection{Final two terms}
%
%%%%%%%%%%%%%%%%%%%%%%%%%%%%%%%%%%%%%%%%%%%%%%%%%%%%
%
 Let us now turn to the last two terms in~\eqref{e:ELIC} where we need Corollary \ref{cor:new}.
Both are bounded in a similar way. We focus only on the last one.
For the other one, we cut out a small part in the middle and then bound the infinite rest as done here. 
Recall that the argument is slightly asymmetric, as we only have a $\mathcal{P}_1$ but a $\mathcal{P}_3^2.$

We have
\[ 
(1-\mathcal{P}_1)(e^{\Delta_\nu T}D)(\varepsilon \cdot)\e_1
=  \mathcal{H} [D(\varepsilon \cdot)]  \cdot e_1,
\]
with convolution operator $ \mathcal{H}_T = H_T\star$ with Fourier transform
\[
 \hat{H}_T(k) = \hat{Q}(k) e^{-4k^2 \varepsilon^{-2} T}e^{\nu T},
\]
where  $\hat{H}_T(0)=0$ and we defined here $\hat{Q}(k)=1-\widehat{P}(k)$, which is slightly different $\hat{Q}$ as defined before, 

but it has the same properties. It is smooth,  has support outside of $[-\delta,\delta]$, and is constant outside $[-2\delta,2\delta]$.
The bounded support is a key point in the argument for this Exchange Lemma, because the $L^2$-norm is not small uniformly in $T$, 
so we need to use Corollary~\ref{cor:new} instead of Lemma~\ref{lem:guidoext}.

Now 

\begin{align*}
 \hat{H}_T'(k) &= \hat{Q}'(k) e^{-4k^2 \varepsilon^{-2} T}e^{\nu T} - \hat{Q}(k)T\varepsilon^{-2}8k  e^{-4k^2 \varepsilon^{-2} T}e^{\nu T},\\
  \hat{H}_T''(k) &= \hat{Q}''(k) e^{-4k^2 \varepsilon^{-2} T}e^{\nu T} -16\hat{Q}'(k)T\varepsilon^{-2}8k  e^{-4k^2 \varepsilon^{-2} T}e^{\nu T} 
  \\&\qquad+ \hat{Q}(k) [ (8T\varepsilon^{-2}k )^2    -T\varepsilon^{-2}8 ] e^{-4k^2 \varepsilon^{-2} T}e^{\nu T}. 
  \end{align*}

Now we use that on the support of $Q'$ we have $|k|\in[\delta,2\delta]$ and the bounds already used many times before, to derive
\[
\begin{split}
 |\hat{H}_T'(k)| & \leqslant   C|\hat{Q}'(k)|  e^{-4\delta^2 \varepsilon^{-2} T} + C|\hat{Q}(k)| T\varepsilon^{-2} k^2 e^{-4k^2 \varepsilon^{-2} T},\\
  |\hat{H}_T''(k)| & \leqslant  C|\hat{Q}''(k)|  e^{-4\delta^2 \varepsilon^{-2} T} +C|\hat{Q}'(k)| \varepsilon^{-2}T   e^{-4 \delta^2 \varepsilon^{-2} T} \\ &\qquad+ C |\hat{Q}(k)| T\varepsilon^{-2}(1+T\varepsilon^{-2}k^2) e^{-2k^2 \varepsilon^{-2} T}. 
  \end{split}
\] 
Thus we can write
\[
\begin{split}
 \|\hat{H}_T'\|^2_{L^2} 
 & \leqslant   C + C T^2\varepsilon^{-4}  \int_{\delta}^{\infty} k^2e^{-8k^2 \varepsilon^{-2} T}dk \\
  & \leqslant   C + C T^{1/2}\varepsilon^{-1}   \int_{\delta T^{1/2} \varepsilon^{-1} }^{\infty} k^2e^{-8k^2}dk  \leqslant  C,
  \end{split}\]
  and similarly
  \[
  \|\hat{H}_T''(k)\|_{L^2} 
  \leqslant  C .
\] 
Using Corollary \ref{cor:new} we obtain
\[
\| (1-\mathcal{P}_1)(e^{\Delta_\nu T}D)(\varepsilon \cdot)\e_1\|_{C_\kappa^0} \leqslant  C \varepsilon^\alpha  \|D\|_{ C_\kappa^{0,\alpha}},
\]
which concludes the proof.
%
%%%%%%%%%%%%%%%%%%%%%%%%%%%%%%%%%%%%%%%%%%%%%%%%%%%%%

\section{Approximation}
\label{sec:app}
%%%%%%%%%%%%%%%%%%%%%%%%%%%%%%%%%%%%%%%%%%%%%%%%%%%%%

In this section we present the proof our main approximation result using the bound on the residual derived in the sections above.
As the result should hold for very long times of order $\varepsilon^{-2}$ we need to rely on the sign of the cubic nonlinearity 
and energy type estimates. But as the Swift-Hohenberg operator does not allow for straightforward $L^p$-estimates, 
we have to restrict the final result to $L^2$-spaces. 

Let us recall the main setting:  $A$ is a mild solution of the amplitude equation~\eqref{e:GL}
and assume that there is a  $\varrho>2$ such that for all $p>0$ one has $A(0)\in W^{1,p}_{\varrho}$,
and $u$ is a solution of the Swift-Hohenberg equation~\eqref{e:SH}.

In order to prove our main result,
we need to bound the error
\[
 R(t) =u(t)-u_A(t) 
\]
between $u$ and the approximation $u_A$ defined in \eqref{def:uA}. Using the definition 
of the residual from \eqref{e:residual} and the mild formulation for the Swift-Hohenberg equation, we obtain
\begin{equation}
 \label{e:derR}
 R(t) = \e^{t\cL_\nu}R(0)  + \int_0^t \e^{(t-s)\cL_\nu} [u_A^3-(u_A+R)^3] ds +  \text{Res}(u_A)(t)\;.
\end{equation}
As the residual $\text{Res}$ is not differentiable in time, 
we cannot proceed with $L^2$-energy estimates as in the deterministic case, but the proof is still very similar.

Substituting $D = R - \text{Res}$,  we obtain first (note that $\text{Res}(0)=0$)
\[
D(t) =  \e^{t\cL_\nu}R(0)  + \int_0^t \e^{(t-s)\cL_\nu} [u_A^3-(D+u_A+\text{Res})^3] ds
\]
and thus
\[
\partial_t D = \cL_\nu D  -(D+u_A+\text{Res}(u_A))^3 - u_A^3 \;.
\]
Now we can use $L^2_{\varrho,\varepsilon}$-energy estimates  
\[
 \frac12\partial_t \| D\|^2_{L^2_{\varrho,\varepsilon}} 
=  \langle \cL_\nu D, D \rangle_{L^2_{\varrho,\varepsilon}} 
- \int_{\R}{ w_{\varrho,\varepsilon}} D[ (D+u_A+\text{Res}(u_A))^3 - u_A^3] \;dx  \;.
\]
We choose the weight (see Definition \ref{def:weight}) for some $\varrho>1$ as
\[
 w_{\varrho,\varepsilon}(x) := \frac1{(1+|\varepsilon x|^2)^{\varrho/2}},
\]
which is integrable with $\| w_{\varrho,\varepsilon}\|_{L^1}=C\varepsilon^{-1}$. 
Also recall that by Lemma \ref{lem:spec}
\[\langle \cL_\nu D, D \rangle_{L^2_{\varrho_\varepsilon}} \leqslant  C \varepsilon^2  \| D\|^2_{L^2_{\varrho,\varepsilon}} \;.
\]
For the nonlinearity we use a straightforward modification of the standard dissipativity result for the cubic in $L^2$-spaces, 
which states that 
\[ \langle (-(D+u_A)^3+D^3, D \rangle_{L^2_\rho} \leq  0 \;.
\]
But here we have the additional term $\text{Res}(u_A)$ that we need to take care of.
Using Young's inequality several times, we obtain
\begin{multline*}
 - [ (D+u_A+\text{Res}(u_A))^3 - u_A^3] \cdot  D \\
 \begin{aligned}
 	= {} & - [  (D+\text{Res})^3+ 3 u_A(D+\text{Res})^2+ 3u_A^2(D+\text{Res})  ] \cdot  D \\
 	= {} & -  D^4 -  3 D^2 \text{Res}^2   - 3u_A^2D^2   \\
 	& - 3 D^3 \text{Res}  - 3 D^3 u_A  -6 D^2u_A\text{Res} 
 	-   D \text{Res}^3- 3D u_A^2 \text{Res}  - 3 D u_A  \text{Res}^2  \\
 	\leqslant  {}&   C  u_A^2 \text{Res}^2   + C \text{Res}^4 \;.
 \end{aligned}
\end{multline*}
The critical terms in the estimates above are:
\[
6 D^2u_A\text{Res} \leqslant  \delta D^2u_A^2 + \delta D^4 + C_\delta \text{Res}^4 
\]
and with $\delta=8/15$ we have
\[
3D^3u_A \leqslant   \frac32  \delta  D^4  + \frac3{2\delta}  D^2u_A^2 = \frac45 D^4  + \frac{45}{16}  D^2u_A^2\;.
\]
In summary we obtain
\[
\partial_t \| D\|^2_{L^2_{\varrho,\varepsilon}} 
\leqslant  
C \varepsilon^2  \| D\|^2_{L^2_{\varrho,\varepsilon}} 
+  C \| u_A\|^2_{L^4_{\varrho,\varepsilon}} \|\text{Res}\|^2_{L^4_{\varrho,\varepsilon}}   
+ C \|\text{Res}\|^4_{L^4_{\varrho,\varepsilon}} \;.
\]
Thus by Gronwall's inequality or comparison principle for ODEs, we obtain directly
\[
\| D(t) \|^2_{L^2_{\varrho,\varepsilon}} 
\leqslant  
e^{C t\varepsilon^2} \| R(0) \|^2_{L^2_{\varrho,\varepsilon}} 
+ \int_0^t e^{C (t-s)\varepsilon^2} \|\text{Res}\|^2_{L^4_{\varrho,\varepsilon}}  (\|\text{Res}\|^2_{L^4_{\varrho,\varepsilon}} + \| u_A\|^2_{L^4_{\varrho,\varepsilon}}) ds
\]
and finally we established the following result.
\begin{lemma}
Let $A$ and $u$ be given as in the beginning of the section.
For the error $R$ given in \eqref{e:derR} we obtain  
\[ \sup_{[0,T_0\varepsilon^{-2}]} \| R-\text{\rm Res}\|_{L^2_{\varrho,\varepsilon}} 
\leqslant  C \| R(0) \|_{L^2_{\varrho,\varepsilon}} 
+   C\varepsilon^{-1} \sup_{[0,T_0\varepsilon^{-2}]} \Big[ \|\text{\rm Res}\|_{L^4_{\varrho,\varepsilon}}  (\|\text{\rm Res}\|_{L^4_{\varrho,\varepsilon}} + \| u_A\|_{L^4_{\varrho,\varepsilon}})   \Big]\;.
\]
\end{lemma}
By assumption, $A(0)\in W^{1,p}_\varrho$ for all $p>0$, so by Corollary \ref{cor:maxregA} we have
\[
	\mathbb{E}\sup_{[0,T]}\|A\|^p_{C^0_\kappa}\leqslant  C_p\qquad \forall p>1, \forall \kappa>0
\]
where $p$ is ``large'' and $\kappa$ is ``small''.

Then, by the Chebychev inequality, we have, in the sense of Definition~\ref{def:0},
\[
	\sup_{[0,T]}\|A\|_{C^0_\kappa}=\mathcal{O}(\varepsilon^{-\delta}), \qquad \forall \delta>0,
\]
and thus
\[
	\sup_{[0,T_0\varepsilon^{-2}]}\|u_A\|_{C^0_{\varrho,\varepsilon}}=\mathcal{O}(\varepsilon^{1-\delta}).
\]

Note that due to the $\varepsilon$ scaling in the weight we have 
\begin{lemma}
 Let $A(0)\in W^{1,p}_\varrho$ for all $p>0$. Then for all $p>1$, $\varrho>1$ and $\delta>0$ 
\[\sup_{[0,T_0\varepsilon^{-2}]}\|u_A\|_{L^p_{\varrho,\varepsilon}} = \mathcal{O} ( \varepsilon^{1 - 1/p-\delta} ).
\]
\end{lemma}

\begin{proof}
	The claim follows from the simple scaling argument below, which is based on a substitution:
	\begin{equation*}
		\|u_A\|_{L^p_{\varrho,\varepsilon}} \leqslant \varepsilon \|A\|_{L^p_{\varrho,\varepsilon}}  =\varepsilon^{1-1/p}\|A\|_{L^p_{\varrho}},
	\end{equation*}
	and we can conclude by noting that $\|A\|_{L^p_{\varrho}}=\mathcal{O}(1)$, with the meaning given in Definition~\ref{def:0}.
\end{proof}
By the result on the residual in Theorem~\ref{thm:res} we have, for all small $\kappa>0$,
\[
	\sup_{[0,T_0\varepsilon^{-2}]}\|\text{Res}(u_A)\|_{C^0_\kappa}=\mathcal{O}(\varepsilon^{3/2-2\kappa}),
\]
thus
\[
	\sup_{[0,T_0\varepsilon^{-2}]}\|\text{Res}(u_A)\|_{L^p_{\varrho,\varepsilon}}=\mathcal{O}(\varepsilon^{3/2-1/p-2\kappa}).
\]

In conclusion,
\[
	\begin{split}
	\sup_{[0,T_0\varepsilon^{-2}]}\|R\|_{L^2_{\varrho,\varepsilon}} 
	& \leqslant  \sup_{[0,T_0\varepsilon^{-2}]}\|R-\text{Res}(u_A)\|_{L^2_{\varrho,\varepsilon}}
	+\sup_{[0,T_0\varepsilon^{-2}]}\|\text{Res}(u_A)\|_{L^2_{\varrho,\varepsilon}}\\
	&\leqslant  C\|R(0)\|_{L^2_{\varrho,\varepsilon}} + \mathcal{O}(\varepsilon^{1-\delta-2\kappa}),
	\end{split}
\]
where we used once more Definition~\ref{def:0} for the $\mathcal{O}(\varepsilon^{1-\delta-\kappa})$ term.
Thus we finished the estimate on the error. Setting $2\kappa=\delta$, we established:
\begin{theorem}\label{thm:final} Let 
	 $A$ be a solution of the amplitude equation \eqref{e:GL}
	 on $[0,T_0]$ such that there is a $\varrho>2$ so that $A(0)\in W^{1,p}_\varrho$ for all $p>1$. 
	 Let $u$ be the solution to the Swift Hohenberg equation~\eqref{e:SH} and $u_A$ the approximation built through $A$, which is defined in \eqref{def:uA}.
	 
	Then for all $\delta>0$, $q>0$ there exists a constant $C_{q,\delta}$ such that
	\[
		\mathbb{P}(\sup_{[0,T_0\varepsilon^{-2}]}\|u-u_A\|_{L^2_{\varrho,\varepsilon}}\leqslant  C\|u(0)-u_A(0)\|_{L^2_{\varrho,\varepsilon}} + C\varepsilon^{1-2\delta})\geqslant 1-C_{q,\delta}\varepsilon^q,
	\]
	where the weight $ w_{\varrho,\varepsilon}(x) = (1+|\varepsilon x|^2)^{-\varrho/2}$ 
	(see Definition \ref{def:weight}) for some $\varrho>1$.
\end{theorem}

%%%%%%%%%%%%%%%%%%%%%%%%%%%%%%%%%%%%%%%%%%%%%%%%%%%%%

\paragraph{Acknowledgments}
L.A.B.~and D.B.~were supported by DFG-funding BL535-9/2 ``Mehrskalenanalyse stochastischer partieller Differentialgleichungen (SPDEs)'', 
and would also like to thank the M.O.P.S.~program for providing a continuous support during the development of this research.

This is a pre-print of an article published in Communications in Mathematical Physics. The final authenticated version is available online at:\\ \url{https://doi.org/10.1007/s00220-019-03573-7}

%%%%%%%%%%%%%%%%%%%%%%%%%%%%%%%%%%%%%%%%%%%%%%%%%%%%%

\bibliographystyle{abbrv}
\bibliography{TBS1bloemker}

\end{document}